\documentclass[12pt,reqno]{amsart}

 \usepackage{mathrsfs}

\usepackage{color}
\usepackage{amsmath, amsthm, amssymb}
\usepackage{amsfonts}
\usepackage[dvips]{epsfig}
\usepackage{graphicx}
\usepackage{caption}
\usepackage{subcaption}
\usepackage[english]{babel}
\usepackage{hyperref}
\usepackage{textcomp}
\usepackage{tikz}
\usepackage{rotating}
\usepackage[utf8]{inputenc}
\usepackage{cite}
\usepackage{amscd}
\usepackage{color}
\usepackage{bm}
\usepackage{enumerate}
\usepackage{amsmath}
\usepackage{bbm}
\usepackage{verbatim}
\usepackage{hyperref}
\usepackage{amstext}
\usepackage{latexsym}

\theoremstyle{plain}
\newtheorem{theorem}{Theorem}[section]
\newtheorem{definition}[theorem]{Definition}

\newtheorem{proposition}[theorem]{Proposition}
\newtheorem{lemma}[theorem]{Lemma}
\newtheorem{remark}[theorem]{Remark}

\numberwithin{theorem}{section}
\numberwithin{equation}{section}

\newcommand{\average}{{\mathchoice {\kern1ex\vcenter{\hrule height.4pt
width 6pt depth0pt} \kern-9.7pt} {\kern1ex\vcenter{\hrule
height.4pt width 4.3pt depth0pt} \kern-7pt} {} {} }}

\renewcommand{\phi}{\varphi}

\newcommand{\be}{\begin{equation}}
\newcommand{\ee}{\end{equation}}

\renewcommand{\epsilon}{\varepsilon}

\title[]{ Nonlinear eigenvalue problems and bifurcation for quasi-linear elliptic operators}
\author{Emmanuel Wend Benedo Zongo and Bernhard Ruf}
\begin{document}
\address{Dip. di Matematica, Via Saldini 50, 20133 Milano, Italy}
\email{wend.zongo@unimi.it}

\address{Dip. di Matematica, Via Saldini 50, 20133 Milano, Italy}
\email{bernhard.ruf@unimi.it}

\maketitle
\begin{abstract}
 In this paper, we analyze an eigenvalue problem for quasi-linear elliptic operators involving homogeneous Dirichlet boundary conditions in a open smooth bounded domain. We show that the eigenfunctions corresponding to the eigenvalues belong to $L^{\infty}$, which implies $C^{1,\alpha}$ smoothness, and the first eigenvalue is simple. Moreover, we investigate the bifurcation results from trivial solutions using the Krasnoselski bifurcation theorem and from infinity using the Leray-Schauder degree. We also show the existence of multiple critical points using variational methods and the Krasnoselski genus.
\end{abstract}
\tableofcontents
\section{Introduction}
\indent\indent Assume $\Omega\subset\mathbb{R}^N$ ($N\geq 2$) is an open bounded domain with smooth boundary $\partial\Omega$. A classical result in the theory of eigenvalue problems guarantees that the problem
\begin{equation}\label{e1}
\left\{
\begin{array}{l}
-\Delta u=\displaystyle \lambda  u~~\text{in $\Omega$},\\
u = \displaystyle 0~~~~~~~~\text{on $\partial\Omega$}
\end{array}
\right.
\end{equation}
possesses a nondecreasing sequence of eigenvalues and a sequence of corresponding eigenfunctions which define a Hilbert basis in $L^2(\Omega)$ [see, \cite{Hen}]. Moreover, it is known that the first eigenvalue of problem (\ref{e1}) is characterized in the variational point of view by,
$$
\lambda^D_1:=\inf_{u\in W^{1,2}_0(\Omega)\backslash\{0\}}\left\{
\frac{\int_{\Omega}|\nabla u|^2~dx}{\int_{\Omega}u^2~dx}\right\}
.$$ 
Suppose that $p>1$ is a given real number and consider the nonlinear eigenvalue problem with Neumann boundary condition
\begin{equation}\label{e2}
\left\{
\begin{array}{l}
-\Delta_p u=\displaystyle \lambda  u~~\text{in $\Omega$},\\
\frac{\partial u}{\partial\nu} = \displaystyle 0~~~~~~~~\text{on $\partial\Omega$}
\end{array}
\right.
\end{equation}
where $\Delta_p u:=\text{div}(|\nabla u|^{p-2}\nabla u)$ stands for the p-Laplace operator and $\lambda\in\mathbb{R}$. This problem was considered in \cite{ma}, and using a direct method in calculus of variations (if $p>2$) or a mountain-pass argument (if $p\in (\frac{2N}{N+2},2)$) it was shown that the set of eigenvalues of problem (\ref{e2}) is exactly the interval $[0,\infty)$. Indeed, it is sufficient to find one positive eigenvalue, say $-\Delta_p u=\lambda u.$ Then a continuous family of eigenvalues can be found by the reparametrization $u=\alpha v,$ satisfying $-\Delta_p v=\mu(\alpha)v$, with $\mu(\alpha)=\frac{\lambda}{\alpha^{p-2}}.$
\\
 \indent In this paper, we consider the so-called $(p,2)$-Laplace operator [see, \cite{ZXA}] with Dirichlet boundary conditions. More precisely, we analyze the following nonlinear eigenvalue problem,
 \begin{equation}\label{e3}
\left\{
\begin{array}{l}
-\Delta_p u-\Delta u=\displaystyle \lambda  u~~\text{in $\Omega$},\\
u = \displaystyle 0~~~~~~~~~~~~~~~~~~\text{on $\partial\Omega$}
\end{array}
\right.
\end{equation}
where $p\in(1,\infty)\backslash\{2\}$ is a real number. We recall that if $1<p<q,$ then $L^q(\Omega)\subset L^p(\Omega)$ and as a consequence, one has $W^{1,q}_0(\Omega)\subset W^{1,p}_0(\Omega).$ 
We will say that $\lambda\in\mathbb{R}$ is an eigenvalue of problem (\ref{e3}) if there exists $u\in W^{1,p}_0(\Omega)\backslash\{0\}$ (if $p>2$ ), $u\in W^{1,2}_0(\Omega)\backslash\{0\}$ (if $1<p<2$) such that 
\begin{equation}\label{v1}
\int_{\Omega}|\nabla u|^{p-2}\nabla u\cdot \nabla v~dx+\int_{\Omega}\nabla u\cdot \nabla v~dx=\lambda\int_{\Omega}u~v~dx, 
\end{equation}
for all $v\in W^{1,p}_{0}(\Omega)$ (if $p>2$), $v\in W^{1,2}_{0}(\Omega)$ (if $1<p<2$). In this case, such a pair $(u,\lambda)$ is called an eigenpair, and $\lambda\in \mathbb{R}$ is called an eigenvalue and $u\in W^{1,p}_0(\Omega)\backslash\{0\}$ is an eigenfunction associated to $\lambda.$ We say that $\lambda$ is a "first eigenvalue", if the corresponding eigenfunction $u$ is positive or negative.
\\
The operator $-\Delta_p-\Delta$ appears in quantum field theory [see, \cite{Ben}], where it arises in the mathematical description of propagation phenomena of solitary waves. We recall that a solitary wave is a wave which propagates without any temporal evolution in shape.
\\
The operator $-\Delta_p-\Delta$ is a special case of the so called $(p,q)$-Laplace operator given by $-\Delta_p-\Delta_q$ which has been widely studied; for some  results related to our studies,  see e.g., [ \cite{BT1,BT2, CD, Ta, MM} ].\\
The main purpose of this work is to study the nonlinear eigenvalue problem (\ref{e3})   when $p>2,$ and $1<p<2$ respectively . In particular, we show in section \ref{S2} that the set of the first eigenvalues is given by the interval $(\lambda_1^D,\infty)$, where $\lambda_1^D$ is the first Dirichlet eigenvalue of the Laplacian. We show that the first eigenvalue of (\ref{e3}) can be obtained variationally, using a Nehari set for $1<p<2$, and a minimization for $p>2.$
Also in the same section, we recall some results of \cite{ma}, \cite{mi} and \cite{MV}.
In section \ref{S3}, we prove that the eigenfunctions associated to $\lambda$ belong to $L^{\infty}(\Omega)$, the first eigenvalue $\lambda_1^D$ of problem (\ref{e3}) is simple and the corresponding eigenfunctions are positive or negative. In addition, in section \ref{S5} we show a homeomorphism property related to $-\Delta_p-\Delta$. \\
In section \ref{S6}, we prove that $\lambda_1^D$ is a bifurcation point for a branch of first eigenvalues from zero if $p>2,$ and $\lambda_1^D$ is a bifurcation point from infinity if $p<2.$ Also the higher Dirichlet eigenvalues $\lambda^D_k$ are bifurcation points (from $0$ if $p>2,$ respectively from infinity if $1<p<2$ ), if the multiplicity of $\lambda^D_k$  is odd.
Finally in section \ref{S}, we prove by variational methods that if $\lambda\in (\lambda^D_k,\lambda^D_{k+1})$ then there exist at least $k$ nonlinear eigenvalues using  Krasnoselski's genus.
In what follows, we denote by $\|.\|_{1,p}$ and $\|.\|_2$ the norms on $W^{1,p}_0(\Omega)$ and $L^2(\Omega)$  defined respectively by $$\|u\|_{1,p}=\left(\int_{\Omega}|\nabla u|^p~dx\right)^{\frac{1}{p}}~~\text{and}~~\|u\|_2=\left(\int_{\Omega}|u|^2~dx\right)^{\frac{1}{2}},~~\text{for all}~~u\in W^{1,p}_0(\Omega),~~u\in L^2(\Omega).$$
We recall the Poincar\'e inequality, i.e., there exists a positive constant $C_p(\Omega)$ such that 
\begin{equation}\label{pc}
\int_{\Omega}|u|^p~dx\leq C_p(\Omega)\int_{\Omega}|\nabla u|^p~dx~~\text{for all $u\in W^{1,p}_0(\Omega)$},~~1<p<\infty.
\end{equation}
\section{The spectrum of the nonlinear problem}\label{S2}
We now begin with the discussion of the properties of the spectrum of the nonlinear eigenvalues problem (\ref{e3}).
\begin{remark}\label{rm1}
\textup{Any $\lambda\leq 0$ is not an eigenvalue of problem (\ref{e3}).}
\end{remark}
\noindent Indeed, suppose by contradiction that $\lambda=0$ is an eigenvalue of equation (\ref{e3}), then relation (\ref{v1}) with $v=u_{0}$ gives $$\int_{\Omega}|\nabla u_{0}|^p~dx+\int_{\Omega}|\nabla u_{0}|^2~dx=0.$$ Consequently $|\nabla u_{0}|=0$, therefore $u_0$ is constant on $\Omega$ and $u_0=0$ on $\Omega$. And this contradicts the fact that $u_0$ is a nontrivial eigenfunction. Hence $\lambda=0$ is not an eigenvalue  of problem $(\ref{e3}).$\\
Now it remains to show that any $\lambda<0$ is not an eigenvalue of (\ref{e3}). Suppose by contradiction that $\lambda<0$ is an eigenvalue of (\ref{e3}), with $u_{\lambda}\in W^{1,p}_0(\Omega)\backslash\{0\}$ the corresponding eigenfunction. The relation (\ref{v1}) with $v=u_{\lambda}$ implies $$0\leq \int_{\Omega}|\nabla u_{\lambda}|^p~dx+\int_{\Omega}|\nabla u_{\lambda}|^2~dx=\lambda\int_{\Omega}u_{\lambda}^2~dx<0.$$  Which yields a contradiction and thus $\lambda<0$ cannot be an eigenvalue of problem (\ref{e3}).
\begin{lemma}\label{l1}
\textup{Any $\lambda\in(0,\lambda^D_1]$ is not an eigenvalue of (\ref{e3}).}
\end{lemma}
\noindent For the proof see also \cite{ma}.
\begin{proof}
Let $\lambda\in(0,\lambda^D_1)$, i.e., $\lambda^D_1>\lambda.$ Let's assume by contradiction that there exists a $\lambda\in(0,\lambda^D_1)$ which is an eigenvalue of (\ref{e3}) with $u_{\lambda}\in W^{1,2}_0(\Omega)\backslash\{0\}$ the corresponding eigenfunction. Letting $v=u_{\lambda}$ in relation (\ref{v1}), we have on the one hand, $$\int_{\Omega}|\nabla u_{\lambda}|^p~dx+\int_{\Omega}|\nabla u_{\lambda}|^2~dx=\lambda\int_{\Omega}u_{\lambda}^2~dx$$
and on the other hand,  
\begin{equation}\label{c01}
\lambda^D_1 \int_{\Omega}u_{\lambda}^2dx\leq \int_{\Omega}|\nabla u_{\lambda}|^2~dx.
\end{equation}
 By subtracting both side of (\ref{c01}) by $ \lambda\displaystyle{\int_{\Omega}}u_{\lambda}^2~dx$, we obtain $$(\lambda^D_1-\lambda)
\int_{\Omega}u^2_{\lambda}~dx\leq \int_{\Omega}|\nabla u_{\lambda}|^2~dx-\lambda\int_{\Omega}u_{\lambda}^2~dx,$$ $$(\lambda^D_1-\lambda)\int_{\Omega}u^2_{\lambda}~dx\leq \int_{\Omega}|\nabla u_{\lambda}|^2~dx-\lambda\int_{\Omega}u_{\lambda}^2~dx+\int_{\Omega}|\nabla u_{\lambda}|^p~dx=0.$$ Therefore $(\lambda^D_1-\lambda)
\displaystyle{\int_{\Omega}}u^2_{\lambda}~dx\leq 0,$ which is a contradiction. Hence, we conclude that $\lambda\in(0,\lambda^D_1)$ is not an eigenvalue of problem (\ref{e3}).  In order to complete the proof of the Lemma \ref{l1} we shall show that $\lambda=\lambda^D_1$ is not an eigenvalue of (\ref{e3}).\\ By contradiction we assume that $\lambda=\lambda^D_1$ is an eigenvalue of (\ref{e3}). So there exists $u_{\lambda^D_1}\in W^{1,2}_0(\Omega)\backslash\{0\}$ such that relation (\ref{v1}) holds true. Letting $v=u_{\lambda^D_1}$ in relation (\ref{v1}), it follows that $$\int_{\Omega}|\nabla u_{\lambda^D_1}|^p~dx+\int_{\Omega}|\nabla u_{\lambda^D_1}|^2~dx=\lambda^D_1
\int_{\Omega}u_{\lambda^D_1}^2~dx.$$ But $\lambda^D_1
\displaystyle{\int_{\Omega}}u_{\lambda^D_1}^2~dx\leq \displaystyle{\int_{\Omega}}|\nabla u_{\lambda^D_1}|^2~dx,$ therefore 
$$\int_{\Omega}|\nabla u_{\lambda^D_1}|^p~dx+\int_{\Omega}|\nabla u_{\lambda^D_1}|^2~dx\leq \int_{\Omega}|\nabla u_{\lambda^D_1}|^2~dx \Rightarrow \int_{\Omega}|\nabla u_{\lambda^D_1}|^p~dx\leq 0.$$ Using relation (\ref{pc}), we have $u_{\lambda^D_1}=0,$ which is a contradiction since $u_{\lambda^D_1}\in W^{1,2}_0(\Omega)\backslash\{0\}.$ So $\lambda=\lambda^D_1$ is not an eigenvalue of (\ref{e3}).
\end{proof}
\begin{theorem}\label{th1}
\textup{ Assume $p\in(1,2)$. Then the set of first eigenvalues of problem $(\ref{e3})$ is given by
$$(\lambda^D_1,\infty),~~\text{where $\lambda^D_1$~\text{denotes the first eigenvalue of} $-\Delta$ \text{on} $\Omega$}.$$}
\end{theorem}
\begin{proof}
Let $\lambda\in (\lambda_1^D,\infty)$,
and define the energy functional $$J_{\lambda}: W^{1,2}_0(\Omega)\rightarrow\mathbb{R}~~
\text{by}~~J_{\lambda}(u)=\int_{\Omega}|\nabla u|^2~dx+\frac{2}{p}\int_{\Omega}|\nabla u|^p~dx-\lambda\int_{\Omega}u^2~dx.$$ 
One shows that $J_{\lambda}\in C^1(W^{1,2}_0(\Omega),\mathbb{R})$~(see,\cite{ZXA}) with its derivatives given by $$\langle J'_{\lambda}(u),v\rangle=2\int_{\Omega}\nabla u\cdot\nabla v~dx+2
\int_{\Omega}|\nabla u|^{p-2}\nabla u\cdot\nabla v~dx-2\lambda\int_{\Omega}u~v~dx~~
,~\forall~v\in W^{1,2}_0(\Omega). $$
Thus we note that $\lambda$ is an eigenvalue of problem (\ref{e3}) if and only if $J_{\lambda}$ possesses a nontrivial critical point. Considering $J_{\lambda}(\rho e_1),$ where $e_1$ is the $L^2$-normalized first eigenfunction of the Laplacian, we see that
$$J_{\lambda}(\rho e_1)\leq \lambda_1^D\rho^2+C\rho^p-\lambda\rho^2\rightarrow-\infty,~~~\text{as}~~\rho\rightarrow+
\infty.$$
Hence, we cannot establish the coercivity of $J_{\lambda}$ on $W^{1,2}_0(\Omega)$ for $p\in(1,2)$, and consequently we cannot use a direct method in calculus of variations in order to determine a critical point of $J_{\lambda}.$ 
To overcome this difficulty, the idea will be to analyze the functional $J_{\lambda}$ on the so called Nehari manifold defined by \begin{eqnarray*}
\mathcal{N}_{\lambda}
:=\left\{u\in W^{1,2}_0(\Omega)\backslash\{0\}~:\int_{\Omega}|\nabla u|^2~dx+\int_{\Omega}|\nabla u|^p~dx=\lambda\int_{\Omega}u^2~dx\right\}.
\end{eqnarray*}
Note that all non-trivial solutions of (\ref{e3}) lie on $\mathcal{N}_{\lambda}.$ On $\mathcal{N}_{\lambda}$ the functional $J_{\lambda}$ takes the form \begin{eqnarray*}
J_{\lambda}(u)&=&\int_{\Omega}|\nabla u|^2~dx+\frac{2}{p}\int_{\Omega}|\nabla u|^p~dx-\lambda\int_{\Omega}u^2~dx\\
&=&(\frac{2}{p}-1)\int_{\Omega}|\nabla u|^p~dx>0.
\end{eqnarray*}
We have seen in Lemma \ref{l1}, that any $\lambda\in (0,\lambda_1^D]$ is not an eigenvalue of problem (\ref{e3}); see also \cite{ma}. It remains to prove the :\\
\textbf{Claim :} Every $\lambda\in (\lambda^D_1,\infty)$ is a first eigenvalue of problem (\ref{e3}).
 Indeed,
we will split the proof of the claim into four steps.
\begin{enumerate}
\item[Step 1.]
Here we will show that $\mathcal{N}_{\lambda}\neq \emptyset$ and every minimizing sequence for $J_{\lambda}$ on $\mathcal{N}_{\lambda}$ is bounded in $W^{1,2}_0(\Omega)$.
Since $\lambda>\lambda^D_1$ there exists $v_{\lambda}\in W^{1,2}_0(\Omega)$ such that $$\int_{\Omega}|\nabla v_{\lambda}|^2~dx<\lambda\int_{\Omega}v^2_{\lambda}~dx.$$  Then there exists $t>0$ such that 
$tv_{\lambda}\in\mathcal{N}_{\lambda}.$ In fact
$$\int_{\Omega}|\nabla(tv_{\lambda})|^2~dx+\int_{\Omega}|\nabla(tv_{\lambda})|^p~dx=\lambda\int_{\Omega}(tv_{\lambda})^2~dx\Rightarrow$$
$$t^2\int_{\Omega}|\nabla v_{\lambda}|^2~dx+t^p\int_{\Omega}|\nabla v_{\lambda}|^p~dx=t^2\lambda\int_{\Omega}v_{\lambda}^2~dx
\Rightarrow$$ 
$$t=\left(\frac{\int_{\Omega}|\nabla v_{\lambda}|^p~dx}{\lambda
\int_{\Omega}v_{\lambda}^2~dx-
\int_{\Omega}|\nabla v_{\lambda}|^2~dx}\right)^{\frac{1}{2-p}}>0.$$ With such $t$ we have $tv_{\lambda}\in\mathcal{N}_{\lambda}$ and $\mathcal{N}_{\lambda}\neq \emptyset.$\\
\\
Note that for $u\in B_r(v_{\lambda}),$ $r>0$ small, the inequality $\lambda\int_{\Omega}|u|^2dx>\int_{\Omega}|\nabla u|^2dx$ remains valid, and then $t(u)u\in\mathcal{N}_{\lambda}$ for $u\in B_r(v_{\lambda}).$ Since $t(u)\in C^1$ we conclude that $\mathcal{N}_{\lambda}$ is a $C^1$-manifold.\\
Let $\{u_k\}\subset \mathcal{N}_{\lambda}$ be a minimizing sequence of $J_{\lambda}|_{\mathcal{N}_{\lambda}}$, i.e. $J_{\lambda}(u_k)\rightarrow m=\displaystyle{\inf_{w\in\mathcal{N}_{\lambda}}}J_{\lambda}(w).$ Then
\begin{equation}\label{e5}
\lambda\int_{\Omega}u^2_k~dx-
\int_{\Omega}|\nabla u_k|^2~dx=\int_{\Omega}|\nabla u_k|^p~dx\rightarrow\left(\frac{2}{p}-1\right)^{-1}m~\text{as $k\rightarrow\infty$}.
\end{equation}
Assume by contradiction that $\{u_k\}$ is not bounded in $W^{1,2}_0(\Omega)$, i.e. $\displaystyle\int_{\Omega}|\nabla u_k|^2~dx\rightarrow\infty$ as $k\rightarrow\infty$. It follows that $\displaystyle{\int_{\Omega}}u^2_k~dx
\rightarrow\infty$ as $k\rightarrow\infty$, thanks to relation (\ref{e5}). We set $v_k=\frac{u_k}{\|u_k\|_2}.$ Since $\displaystyle{\int_{\Omega}}|\nabla u_k|^2~dx<\lambda\displaystyle{\int_{\Omega}}u_k^2~dx
$, we deduce that $\displaystyle{\int_{\Omega}}|\nabla v_k|^2~dx<\lambda,$ for each $k$ and  $\|v_k\|_{1,2}<\sqrt{\lambda}.$ Hence $\{v_k\}\subset W^{1,2}_0(\Omega)$ is bounded in $W^{1,2}_0(\Omega).$ Therefore there exists $v_0\in W^{1,2}_0(\Omega)$ such that $v_k\rightharpoonup v_0$ in $W^{1,2}_0(\Omega)\subset W^{1,p}_0(\Omega)$ and $v_k\rightarrow v_0$ in $L^2(\Omega).$ Dividing relation (\ref{e5}) by $\|u_k\|^p_{2}$, we get $$\int_{\Omega}|\nabla v_k|^p~dx=\frac{\lambda\displaystyle{\int_{\Omega}}u^2_k~dx-
\displaystyle{\int_{\Omega}}|\nabla u_k|^2~dx}{\|u_k\|^p_{2}}\rightarrow 0~~\text{as~k $\rightarrow\infty$},$$ since $\lambda\displaystyle{\int_{\Omega}}u^2_k~dx-
\displaystyle{\int_{\Omega}}|\nabla u_k|^2~dx\rightarrow \left(\frac{2}{p}-1\right)^{-1}m<\infty$ and $\|u_k\|^p_{2}\rightarrow\infty$ as $k\rightarrow\infty$. On the other hand, since $v_k\rightharpoonup v_0$ in $W^{1,p}_0(\Omega)$, we have $\displaystyle{\int_{\Omega}}|\nabla v_0|^p~dx\leq \lim_{k\rightarrow\infty}\inf\displaystyle
{\int_{\Omega}}|\nabla v_k|^p~dx=0$ and consequently $v_0=0$. It follows that $v_k\rightarrow 0$ in $L^2(\Omega),$ which is a contradiction since $\|v_k\|_{2}=1$. Hence, $\{u_k\}$ is bounded in $W^{1,2}_0(\Omega).$\\
\item[Step 2.] $m=
\displaystyle{\inf_{w\in\mathcal{N}_{\lambda}}}J_{\lambda}(w)>0.$ Indeed,
assume by contradiction that $m=0$. Then, for $\{u_k\}$ as in step 1, we have

\begin{equation}\label{e6}
0<\lambda\int_{\Omega}u^2_k~dx-
\int_{\Omega}|\nabla u_k|^2~dx=\int_{\Omega}|\nabla u_k|^p~dx\rightarrow 0, \text{as $k\rightarrow\infty$}.
\end{equation}
By Step 1, we deduce that $\{u_k\}$ is bounded in $W^{1,2}_0(\Omega).$ Therefore there exists $u_0\in W^{1,2}_0(\Omega)$ such that $u_k \rightharpoonup u_0$ in $W^{1,2}_0(\Omega)$ and $W^{1,p}_0(\Omega)$ and $u_k\rightarrow u_0$ in $L^2(\Omega).$

Thus $\displaystyle{\int_{\Omega}}|\nabla u_0|^p~dx\leq \lim_{k\rightarrow\infty}\inf\displaystyle
{\int_{\Omega}}|\nabla u_k|^p~dx=0.$ And consequently  $u_0=0$, $u_k \rightharpoonup 0$ in $W^{1,2}_0(\Omega)$ and $W^{1,p}_0(\Omega)$ and $u_k\rightarrow 0$ in $L^2(\Omega).$ Writing again $v_k=\frac{
u_k}{\|u_k\|_2}$ we have
 $$0<\frac{\lambda\displaystyle{\int_{\Omega}}u^2_k~dx-\int_{\Omega}|\nabla u_k|^2~dx}{\|u_k\|^2_{2}}=\|u_k\|^{p-2}_{2} \displaystyle{\int_{\Omega}}|\nabla v_k|^p~dx,$$ therefore
\begin{eqnarray*}
\int_{\Omega}|\nabla v_k|^p~dx
&=&\|u_k\|^{2-p}_{2}\left(\frac{\lambda\|u_k\|^2_2}{\|u_k\|^2_{2}}-\frac{\displaystyle{\int_{\Omega}}|\nabla u_k|^2~dx}{\|u_k\|^2_{2}}\right)\\
&=&\|u_k\|^{2-p}_{2}\left(\lambda-
\displaystyle{\int_{\Omega}}|\nabla v_k|^2~dx\right)\rightarrow 0~ \text{as $k\rightarrow \infty$},
\end{eqnarray*}
since $\|u_k\|_{2}\rightarrow 0$ and $p\in(1,2)$, and $\{v_k\}$ is bounded in $W_0^{1,2}(\Omega).$ Next since $v_k \rightharpoonup v_0$ in $W^{1,2}_0(\Omega)\subset W^{1,p}_0(\Omega)$, we deduce that $\displaystyle{\int_{\Omega}}|\nabla v_0|^p~dx\leq \lim_{k\rightarrow\infty}\inf\displaystyle
{\int_{\Omega}}|\nabla v_k|^p~dx=0$ and we have $v_0=0.$ And it follows that $v_k\rightarrow 0$ in $L^2(\Omega)$ which is a contradiction since $\|v_k\|_{2}=1$ for each k. Hence, $m$ is positive.
\item[Step 3.]
There exists $u_0\in \mathcal{N}_{\lambda}$ such that $J_{\lambda}(u_0)=m.$

Let $\{u_k\}\subset\mathcal{N}_{\lambda}$ be a minimizing sequence, i.e., $J_{\lambda}(u_k)\rightarrow m$ as $k\rightarrow\infty.$ Thanks to Step 1, we have that $\{u_k\}$ is bounded in $W_0^{1,2}(\Omega).$ It follows that there exists $u_0\in W_0^{1,2}(\Omega)$ such that $u_k\rightharpoonup u_0$ in $W_0^{1,2}(\Omega)$ and $W_0^{1,p}(\Omega)$ and strongly in $L^2(\Omega).$ The results in the two steps above guarantee that $J_{\lambda}(u_0)\leq \displaystyle{\lim_{k\rightarrow\infty}}\inf J_{\lambda}(u_k)=m.$ Since for each $k$ we have $u_k\in\mathcal{N}_{\lambda}$, we have
\begin{equation}\label{Z3}
\int_{\Omega}|\nabla u_k|^2~dx+\int_{\Omega}|\nabla u_k|^p~dx=\lambda \int_{\Omega}u^2_k~dx~~~\text{for all $k$.}
\end{equation}
Assuming $u_0\equiv 0$ on $\Omega$ implies that $ \displaystyle{\int_{\Omega}}u^2_k~dx\rightarrow 0$ as $k\rightarrow \infty$, and by relation $(\ref{Z3})$ we obtain that $\displaystyle{\int_{\Omega}}|\nabla u_k|^2~dx\rightarrow 0$ as $k\rightarrow \infty.$ Combining this with the fact that $u_k$ converges weakly to $0$ in $W_0^{1,2}(\Omega)$, we deduce that $u_k$ converges strongly to $0$ in $W_0^{1,2}(\Omega)$ and consequently in $W_0^{1,p}(\Omega)$. Hence we infer that \begin{eqnarray*}
\lambda\int_{\Omega}u^2_k~dx-
\int_{\Omega}|\nabla u_k|^2~dx=\int_{\Omega}|\nabla u_k|^p~dx\rightarrow 0, \text{as $k\rightarrow\infty$}.
\end{eqnarray*}
Next, using similar argument as the one used in the proof of Step 2, we will reach to a contradiction, which shows that $u_0\not\equiv 0.$ Letting $k\rightarrow \infty$ in relation (\ref{Z3}), we deduce that 
\begin{eqnarray*}
\int_{\Omega}|\nabla u_0|^2~dx+\int_{\Omega}|\nabla u_0|^p~dx\leq\lambda \int_{\Omega}u_0^2~dx.
\end{eqnarray*}
If there is equality in the above relation then $u_0\in\mathcal{N}_{\lambda}$ and $m\leq J_{\lambda}(u_0)$. Assume by contradiction that
\begin{equation}\label{Z4}
\int_{\Omega}|\nabla u_0|^2~dx+\int_{\Omega}|\nabla u_0|^p~dx<\lambda \int_{\Omega}u^2_0~dx.
\end{equation} 
Let $t>0$ be such that $tu_0\in\mathcal{N}_{\lambda},$ i.e.,$$t=\left(\frac{\lambda\displaystyle{\int_{\Omega}}u_0^2~dx-
\displaystyle{\int_{\Omega}}|\nabla u_0|^2~dx}{\displaystyle{\int_{\Omega}}|\nabla u_0|^p~dx}\right)^{\frac{1}{p-2}}.$$
We note that $t\in(0,1)$ since $1<t^{p-2}$ (thanks to (\ref{Z4})). Finally, since $tu_0\in \mathcal{N}_{\lambda}$ with $t\in(0,1)$ we have 
\begin{eqnarray*}
0<m\leq J_{\lambda}(tu_0)&=&\left(\frac{2}{p}-1\right)\int_{\Omega}|\nabla(tu_0)|^p~dx=t^p\left(\frac{2}{p}-1\right)\int_{\Omega}|\nabla u_0|^p~dx\\
&=&t^p J_{\lambda}(u_0)\\
&\leq & t^p\lim_{k\rightarrow\infty}\inf J_{\lambda}(u_k)=t^p m<m~\text{for $t\in(0,1)$,}
\end{eqnarray*}
and this is a contradiction which assures that relation (\ref{Z4}) cannot hold and consequently we have $u_0\in \mathcal{N}_{\lambda}$. Hence $m\leq J_{\lambda}(u_0)$ and $ m= J_{\lambda}(u_0)$.
\item[Step 4.] We conclude the proof of the claim.
Let $u\in \mathcal{N}_{\lambda}$ be such that $J_{\lambda}(u)=m$ (thanks to Step 3). Since $u\in \mathcal{N}_{\lambda}$, we have \begin{eqnarray*}
\int_{\Omega}|\nabla u|^2~dx+\int_{\Omega}|\nabla u|^p~dx=\lambda \int_{\Omega}u^2~dx,
\end{eqnarray*} 
and $$\int_{\Omega}|\nabla u|^2~dx<\lambda \int_{\Omega}u^2~dx.$$
Let $v\in\partial B_1(0)\subset W_0^{1,2}(\Omega)$ and $\varepsilon>0$ be very small such that $u+\delta v\ne 0$ in $\Omega$ for all $\delta\in(-\varepsilon,\varepsilon)$ and $$\int_{\Omega}|\nabla(u+\delta v)|^2~dx<\lambda\int_{\Omega}(u+\delta v)^2~dx ;$$ this is equivalent to 
$$\lambda\int_{\Omega}u^2~dx-\int_{\Omega}|\nabla u|^2~dx>\delta\left(2\int_{\Omega}\nabla u\cdot\nabla v~dx-
2\lambda\int_{\Omega}u v~dx\right)+\delta^2\left(\int_{\Omega}|\nabla v|^2~dx-\lambda\int_{\Omega}v^2~dx\right),$$ which holds true for $\delta$ small enough since the left hand side is positive while the function $$h(\delta):=|\delta|\left|2\int_{\Omega}\nabla u\cdot\nabla v~dx-
2\lambda\int_{\Omega}u v~dx\right|+\delta^2\left|\int_{\Omega}|\nabla v|^2~dx-\lambda\int_{\Omega}v^2~dx\right|$$ dominates the term from the right hand side and $h(\delta)$ is a continuous function (polynomial in $\delta$) which vanishes in $\delta=0$. For each $\delta\in(-\varepsilon,\varepsilon),$ let $t(\delta)>0$ be given by 
$$t(\delta)=\left(\frac{\lambda\displaystyle{\int_{\Omega}}(u+\delta v)^2~dx-
\int_{\Omega}|\nabla(u+\delta v)|^2~dx}{\displaystyle{\int_{\Omega}}|\nabla(u+\delta v)|^p~dx}\right)^{\frac{1}{p-2}},$$ so that $t(\delta)\cdot(u+\delta v)\in \mathcal{N}_{\lambda}.$ We have that $t(\delta)$ is of class $C^1(-\varepsilon,\varepsilon)$ since $t(\delta)$ is the composition of some functions of class $C^1$. On the other hand, since $u\in \mathcal{N}_{\lambda}$ we have $t(0)=1.$

Define $\iota:(-\varepsilon,\varepsilon)\rightarrow\mathbb{R}$~~by $\iota(\delta)=J_{\lambda}(t(\delta)(u+\delta v))$ which is of class $C^1(-\varepsilon,\varepsilon)$ and has a minimum at $\delta=0.$ We have 
$$\iota'(\delta)=[t'(\delta)(u+\delta v)+v t(\delta)]J'_{\lambda}(t(\delta)(u+\delta v))\Rightarrow$$ 
$$0=\iota'(0)=J'_{\lambda}(t(0)(u))[t'(0)u+v t(0)]=\langle J'_{\lambda}(u),v\rangle$$ since $t(0)=1$ and $t'(0)=0.$\\ This shows that every $\lambda\in(\lambda^D_1,\infty)$ is an eigenvalue of problem (\ref{e3}). 
\end{enumerate}
\end{proof}
\noindent In the next theorem we consider the case $p>2.$ For similar results for the Neumann case, [see, \cite{mi}].
\begin{theorem}\label{th2} \textup{For $p>2$, the set of first eigenvalues of problem (\ref{e3}) is given by $(\lambda^D_1,\infty).$}
\end{theorem}
\noindent The proof of Theorem \ref{th2} will follow as a direct consequence of the lemmas proved below:
\begin{lemma}
\textup{Let $$\lambda_1(p):=\inf_{u\in W^{1,p}_0
\backslash\{0\}}\left\{\frac{\frac{1}{p}\displaystyle{\int_{\Omega}}|\nabla u|^p~dx+\frac{1}{2}\int_{\Omega}|\nabla u|^2~dx}{\frac{1}{2}\displaystyle{\int_{\Omega}}u^2~dx}\right\}.$$ Then $\lambda_1(p)=\lambda^D_1,$ for all $p>2.$}
\end{lemma}
\begin{proof}
We clearly have $\lambda_1(p)\geq \lambda^D_1$ since a positive term is added. On the other hand, consider $u_n=\frac{1}{n} e_1$ (where $e_1$ is the first eigenfunction of $-\Delta$), we get
$$\lambda_1(p)\leq \frac{\frac{1}{2 n^2}\int_{\Omega}|\nabla e_1|^2dx+\frac{1}{p n^p}\int_{\Omega}|\nabla e_1|^p dx}{\frac{1}{2 n^2}\int_{\Omega}|e_1|^2dx}\rightarrow \lambda^D_1~~\text{as $n\rightarrow \infty$}.$$
\end{proof}
\begin{lemma}\label{l3}
\textup{ For each $\lambda>0$, we have
 $$\lim_{\|u\|_{1,p}\rightarrow\infty}\left(\frac{1}{2}\int_{\Omega}|\nabla u|^2~dx+\frac{1}{p}\int_{\Omega}|\nabla u|^p~dx-\frac{\lambda}{2}\int_{\Omega}u^2~dx\right)=\infty.$$}
\end{lemma}
\begin{proof}
Clearly
 \begin{eqnarray*} 
 \frac{1}{p}\int_{\Omega}|\nabla u|^p~dx+\frac{1}{2}\int_{\Omega}|\nabla u|^2~dx &\geq& \frac{1}{p}\int_{\Omega}|\nabla u|^p~dx.\\
 \end{eqnarray*} 
 On the one hand, using Poincar\'e's inequality with $p=2$,\\ we have $\displaystyle{\int_{\Omega}}u^2~dx\leq C_2(\Omega)\displaystyle{\int_{\Omega}}|\nabla u|^2~dx, \forall u\in W^{1,p}_0(\Omega)\subset W^{1,2}_0(\Omega)$ and then applying the H\"older inequality to the right hand side term of the previous estimate, we obtain
 \begin{eqnarray*}
 \int_{\Omega}|\nabla u|^2~dx 
 \leq  |\Omega|^{\frac{p-2}{p}}\|u\|^2_{1,p},
 \end{eqnarray*}
 so $\displaystyle{\int_{\Omega}}u^2~dx\leq D \|u\|^2_{1,p},$ where $D=C_2(\Omega) |\Omega|^{\frac{p-2}{p}}.$ Therefore for $\lambda>0,$
 \begin{equation}\label{est}
 \frac{1}{2}\int_{\Omega}|\nabla u|^2~dx+\frac{1}{p}\int_{\Omega}|\nabla u|^p~dx-\frac{\lambda}{2}\int_{\Omega}u^2~dx\geq C\|u\|^p_{1,p}-\frac{\lambda}{2} D \|u\|^2_{1,p},
 \end{equation}
and the the right-hand side of (\ref{est}) tends to $\infty$, as $\|u\|_{1,p}\rightarrow\infty$, since $p>2.$
 \end{proof}
 \begin{lemma}
 
 \textup{Every $\lambda\in (\lambda^D_1,\infty)$ is a first eigenvalue of problem (\ref{e3}).}
 \end{lemma}
 \begin{proof}
 For each $\lambda>\lambda^D_1$ define $F_{\lambda}: W^{1,p}_0(\Omega)\rightarrow \mathbb{R}$ by $$F_{\lambda}(u)=\frac{1}{2}\int_{\Omega}|\nabla u|^2~dx+\frac{1}{p}\int_{\Omega}|\nabla u|^p~dx-\frac{\lambda}{2}\int_{\Omega}u^2~dx~~,\forall u\in W^{1,p}_0(\Omega).$$ Standard arguments shows that $F_{\lambda}\in C^1(W^{1,p}_0(\Omega),\mathbb{R})$ [see, \cite{ZXA}] with its derivative given by $$\langle F'_{\lambda}(u),\phi\rangle=\int_{\Omega}(|\nabla u|^{p-2}+1)\nabla u\cdot \nabla\phi~dx-\lambda\int_{\Omega}u\phi~dx,$$ for all $u,\phi\in W^{1,p}_0(\Omega).$ Estimate (\ref{est}) shows that $F_{\lambda}$ is coercive in $W^{1,p}_0(\Omega).$ On the other hand, $F_{\lambda}$ is also weakly lower semi-continuous on $W^{1,p}_0(\Omega)$ since $F_{\lambda}$ is a continuous convex functional, (see  \cite[Proposition 1.5.10 and Theorem 1.5.3]{ME}) . Then we can apply a calculus of variations result, in order to obtain the existence of a global minimum point of $F_{\lambda},$ denoted by $\theta_{\lambda},$ i.e., $F_{\lambda}(\theta_{\lambda})=\displaystyle{\min_{W^{1,p}_0(\Omega)}}F_{\lambda}.$
 Note that for any $\lambda> \lambda^D_1$ there exists $u_{\lambda}\in W^{1,p}_0(\Omega)$ such that $F_{\lambda}(u_{\lambda})< 0$ . Indeed, taking $u_{\lambda}=re_1,$ we have $$F_{\lambda}(re_1)=\frac{r^2}{2}(\lambda_1^D-\lambda)+\frac{r^p}{p}\int_{\Omega}|\nabla e_1|^p~dx<0~~~\text{for}~~r>0~~\text{small}.$$
But then $F_{\lambda}(\theta_{\lambda})\leq F_{\lambda}(u_{\lambda})< 0$, which means that $\theta_{\lambda}\in W^{1,p}_0(\Omega)\backslash\{0\}.$ 
On the other hand, we have $\langle F'_{\lambda}(\theta_{\lambda}),\phi\rangle=0,\forall\phi\in W^{1,p}_0(\Omega)$ ($\theta_{\lambda}$ is a critical point of $F_{\lambda}$) with $\theta_{\lambda}\in W^{1,p}_0(\Omega)\backslash\{0\}\subset W^{1,2}_0(\Omega) \backslash\{0\}.$ Consequently each $\lambda> \lambda^D_1$ is an eigenvalue of problem (\ref{e3}).
\end{proof}
\begin{proposition}\label{N2}
The first eigenfunctions $u^{\lambda}_1$ associated to $\lambda\in (\lambda^D_1,\infty)$
are positive or negative in $\Omega.$
\end{proposition}
\begin{proof}
Let $u^{\lambda}_1\in W^{1,p}_0(\Omega)\backslash\{0\}$ be an eigenfunction associated to $\lambda\in (\lambda^D_1,\infty)$, then \\$\displaystyle{\int_{\Omega}|\nabla u^{\lambda}_1|^p~dx+\displaystyle{\int_{\Omega}}|\nabla u^{\lambda}_1|^2~dx=\lambda\int_{\Omega}|u^{\lambda}_1|^2~dx},$ which means $u^{\lambda}_1$ achieves the infimum in the definition of $\lambda^D_1$. On the other hand we have $\|\nabla|u^{\lambda}_1|\|_{1,p}=\|\nabla u^{\lambda}_1 \|_{1,p}$ and $\|\nabla|u^{\lambda}_1|\|_{1,2}=\|\nabla u^{\lambda}_1 \|_{1,2}$, since $|\nabla |u^{\lambda}_1||=|\nabla u^{\lambda}_1|$ almost everywhere. It follows that $|u^{\lambda}_1|$ achieves also the infimum in the definition of $\lambda^D_1$, and therefore by the Harnack inequality [see, \cite{LE}], we have $|u^{\lambda}_1(x)|> 0$ $\forall x\in\Omega$ and consequently $u^{\lambda}_1$ is either positive or negative in $\Omega$.
\end{proof}
\noindent A similar result of Theorem \ref{T1} was proved in \cite{A} in the case of the $p$-Laplacian.
\section{Properties of eigenfunctions and the operator $-\Delta_p-\Delta$}\label{S3}
\subsection{Boundedness of the eigenfunctions}
We shall prove boundedness of eigenfunctions and use this fact to obtain $C^{1,\alpha}$ smoothness of all eigenfunctions of the quasi-linear problem (\ref{e3}). The latter result is due to \cite[Theorem 4.4]{A}, which originates from \cite{E} and \cite{P}.
\begin{theorem}\label{T1}
\textup{Let $(u,\lambda)\in W^{1,p}_0(\Omega)\times \mathbb{R}^{\star}_+$ be an eigensolution of the weak formulation (\ref{v1}). Then $u\in L^{\infty}(\Omega).$}
\end{theorem}
\begin{proof}
By Morrey's embedding theorem it suffices to consider the case $p\leq N.$ Let us assume first that $u>0.$ For $M\geq 0$ define $w_M(x)=\min\{u(x),M\}$. Letting  
\begin{equation}
g(x)=
\left\{
\begin{array}{l}
x~~\text{if~$x\leq M$}\\
M~~\text{if~$x>M$}
\end{array}
\right.
\end{equation}
we have $g\in C(\mathbb{R})$ piecewise smooth function with $g(0)=0.$ Since $u\in W^{1,p}_0(\Omega)$ and $g'\in L^{\infty}(\Omega),$ then $g\circ u\in  W^{1,p}_0(\Omega)$ and $w_M\in W^{1,p}_0(\Omega)\cap L^{\infty}(\Omega)$ (see, Theorem B.3 in \cite{A}). For $k>0,$ define $\varphi=w^{kp+1}_M$ then $\nabla\varphi=(kp+1)\nabla w_Mw^{kp}_M$ and $\varphi\in  W^{1,p}_0(\Omega)\cap L^{\infty}(\Omega).$\\ Using $\varphi$ as a test function in (\ref{v1}), one obtains $$(kp+1)\left[\int_{\Omega}|\nabla u|^{p-2}\nabla u\cdot \nabla w_M w^{kp}_M~dx+\int_{\Omega}\nabla u\cdot\nabla w_M w^{kp}_M~dx\right]=\lambda\int_{\Omega}u~
w^{kp+1}_M~dx.$$ On the other hand using the fact that $w^{kp+1}_M\leq u^{kp+1}$, it follows that 
$$(kp+1)\left[\int_{\Omega}|\nabla u|^{p-2}\nabla u\cdot \nabla w_M w^{kp}_M~dx+\int_{\Omega}\nabla u\cdot\nabla w_M w^{kp}_M~dx\right]\leq\lambda\int_{\Omega}|u|^{(k+1)p}~dx.$$ We have $\nabla(w^{k+1}_M)=(k+1)\nabla w_M w^k_M\Rightarrow |\nabla w^{k+1}_M|^p=(k+1)^pw^{kp}_M|\nabla w_M|^p.$  Since the integrals on the left are zero on $\{x : u(x)>M\}$ we can take $u=w_M$ in the previous inequality, and it follows that
$$(kp+1)\left[\int_{\Omega}|\nabla w_M|^p w^{kp}_M~dx+\int_{\Omega}|\nabla w_M|^2 w^{kp}_M~dx\right]\leq\lambda\int_{\Omega}|u|^{(k+1)p}~dx.$$ Replacing $|\nabla w_M|^p w^{kp}_M$~~by $\frac{1}{(k+1)^p}|\nabla w^{k+1}_M|^p$, we have $$\frac{kp+1}{(k+1)^p}\int_{\Omega}|\nabla w^{k+1}_M|^p~dx+(kp+1)\int_{\Omega}|\nabla w_M|^2 w^{kp}_M~dx\leq\lambda\int_{\Omega}|u|^{(k+1)p}~dx,$$ 
 
\noindent which implies that $$\frac{kp+1}{(k+1)^p}\int_{\Omega}|\nabla w^{k+1}_M|^p~dx\leq\lambda\int_{\Omega}|u|^{(k+1)p}~dx~$$ and  then 
\begin{equation}\label{ast}
\int_{\Omega}|\nabla w^{k+1}_M|^p~dx\leq\left(\lambda \frac{(k+1)^p}{kp+1}\right)\int_{\Omega}|u|^{(k+1)p}~dx.
\end{equation}

By Sobolev's embedding theorem, there is a constant $c_1>0$ such that 
\begin{equation}\label{sbe}
\|w^{k+1}_M\|_{p^{\star}}\leq c_1\|w^{k+1}_M\|_{1,p}.
\end{equation}
where $p^{\star}$ is the Sobolev critical exponent.
Consequently, we have
\begin{equation}\label{sb1}
 \|w_M\|_{(k+1)p^{\star}}\leq \|w^{k+1}_M\|^{\frac{1}{k+1}}_{p^{\star}},
\end{equation} 
and therefore 
\begin{equation}\label{sb2}
\|w_M\|_{(k+1)p^{\star}}\leq \left(c_1\|w^{k+1}_M\|_{1,p}\right)^{\frac{1}{k+1}}=c_1^{\frac{1}{k+1}}\|w^{k+1}_M\|^{\frac{1}{k+1}}_{1,p}.
\end{equation} 
But by $(\ref{ast})$,
\begin{equation}\label{o}
\|w^{k+1}_M\|_{1,p}\leq \left(\lambda \frac{(k+1)^p}{kp+1}\right)^{\frac{1}{p}}\|u\|^{k+1}_{(k+1)p}
\end{equation} 
and we note that we can find a constant $c_2>0$ such that \\$\left(\lambda \frac{(k+1)^p}{kp+1}\right)^{\frac{1}{p\sqrt{k+1}}}\leq c_2$, independently of $k$ and consequently
\begin{equation}\label{q}
 \|w_M\|_{(k+1)p^{\star}}\leq c_1^{\frac{1}{k+1}}c_2^{\frac{1}{\sqrt{k+1}}}\|u\|_{(k+1)p}.
 \end{equation}
  Letting $M\rightarrow\infty$, Fatou's lemma implies
 \begin{equation}\label{r} 
\|u\|_{(k+1)p^{\star}}\leq  c_1^{\frac{1}{k+1}}c_2^{\frac{1}{\sqrt{k+1}}}\|u\|_{(k+1)p}.
\end{equation}
 Choosing $k_1,$ such that $(k_1+1)p=p^{\star},$ then $\|u\|_{(k_1+1)p^{\star}}\leq  c_1^{\frac{1}{k_1+1}}c_2^{\frac{1}{\sqrt{k_1+1}}}\|u\|_{p^{\star}}$. Next we choose $k_2$ such that $(k_2+1)p=(k_1+1)p^{\star},$ then taking $k_2=k$ in inequality (\ref{r}), it follows that 
 \begin{equation}\label{zg1}
 \|u\|_{(k_2+1)p^{\star}}\leq  c_1^{\frac{1}{k_2+1}}c_2^{\frac{1}{\sqrt{k_2+1}}}\|u\|_{(k_1+1)p^{\star}}.
 \end{equation}
 By induction we obtain 
\begin{equation}\label{s}
\|u\|_{(k_n+1)p^{\star}}\leq  c_1^{\frac{1}{k_n+1}}c_2^{\frac{1}{\sqrt{k_n+1}}}\|u\|_{(k_{n-1}+1)p^{\star}},
\end{equation}
 where the sequence $\{k_n\}$ is chosen such that $(k_n+1)p=(k_{n-1}+1)p^{\star}, k_0=0.$ One gets $k_n+1=(\frac{p^{\star}}{p})^n.$ As $\frac{p}{p^{\star}}<1,$ there is $C>0$ (which depends on $c_1$ and $c_2$) such that for any $n=1,2,\cdots$
 \begin{equation}\label{che1}
 \|u\|_{r_n}\leq C\|u\|_{p^{\star}}
 \end{equation}
with $r_n=(k_n+1)p^{\star}\rightarrow\infty$ as $n\rightarrow\infty.$ We note that (\ref{che1}) follows by iterating the previous inequality (\ref{s}).
We will indirectly show that $u\in L^{\infty}(\Omega).$ Suppose $u\not\in L^{\infty}(\Omega),$ then there exists $\varepsilon>0$ and a set $A$ of positive measure in $\Omega$ such that 
 $|u(x)|>C \|u\|_{p^{\star}}+\varepsilon=K,$ for all $x\in A.$ We then have, 
 \begin{equation}\label{sbb1}
\lim_{n\rightarrow\infty}\inf\|u\|_{r_n}\geq \lim_{n\rightarrow\infty}\inf\left(\int_{A}
K^{r_n}\right)^{1/r_n}=\lim_{n\rightarrow
\infty}\inf K |A|^{1/r_n}=K>C \|u\|_{p^{\star}},
\end{equation}
which contradicts $(\ref{che1}).$
If $u$ changes sign, we consider $u=u^+ -u^-$ where
\begin{equation}\label{dec}
u^+=\max\{u,0\}~~\text{and}~~u^-=\max\{-u,0\}.
\end{equation}
 We have $u^+,u^-\in W^{1,p}_0(\Omega).$ For each $M>0$ define $w_M=\min\{u^+(x),M\}$ and take again $\varphi=w^{kp+1}_M$ as a test function in $(\ref{v1}).$ Proceeding the same way as above we conclude that $u^+\in L^{\infty}(\Omega)$. Similarly we have $u^-\in L^{\infty}(\Omega)$. Therefore $u=u^+ -u^-$ is in $L^{\infty}(\Omega).$
\end{proof}
\subsection{Simplicity of the eigenvalues}\label{S4}
\noindent We prove an auxiliary result which will imply uniqueness of the first eigenfunction.
\noindent Let 
\begin{eqnarray*}
I(u,v)&=&\langle -\Delta_p u,\frac{u^p-v^p}{u^{p-1}}\rangle+\langle -\Delta u,\frac{u^2-v^2}{u}\rangle\\&+&\langle -\Delta_p v,\frac{v^p-u^p}{v^{p-1}}\rangle+\langle -\Delta v,\frac{v^2-u^2}{v}\rangle,\nonumber
\end{eqnarray*}
for all $(u,v)\in D_I,$ where 
\begin{equation*}
D_I=\{(u_1,u_2)\in W^{1,p}_0(\Omega)\times W^{1,p}_0(\Omega)~:~u_i>  0 ~\text{in $\Omega$}~\text{and}~u_i\in L^{\infty}(\Omega)~~\text{for}~i=1,2\}~~\text{if $p>2$},
\end{equation*}
and 
\begin{equation*}
D_I=\{(u_1,u_2)\in W^{1,2}_0(\Omega)\times W^{1,2}_0(\Omega)~:~u_i>  0 ~\text{in $\Omega$}~\text{and}~u_i\in L^{\infty}(\Omega)~~\text{for}~i=1,2\}~~\text{if $1<p<2$}.
\end{equation*}
\begin{proposition}\label{p1}
\textup{For all $(u,v)\in D_I$, we have $I(u,v)\geq0.$ Furthermore $I(u,v)=0$ if and only if there exists $\alpha\in \mathbb{R}^{\star}_+$ such that $u=\alpha v.$}
\end{proposition}
\begin{proof}
 We first show that $I(u,v)\geq 0.$ We recall that (if $2<p<\infty$) 
 \begin{equation*}
 \langle -\Delta_p u,w\rangle=\int_{\Omega}|\nabla u|^{p-2}\nabla u\cdot\nabla w~dx~~~\text{for all $w\in W^{1,p}_0(\Omega)$}
 \end{equation*}
 
\begin{equation*}
\langle -\Delta u,w\rangle=\int_{\Omega}\nabla u\cdot\nabla w~dx~~~\text{for all $w\in W^{1,p}_0(\Omega)$}.
\end{equation*}
and (if $1<p<2$)
\begin{equation*}
 \langle -\Delta_p u,w\rangle=\int_{\Omega}|\nabla u|^{p-2}\nabla u\cdot\nabla w~dx~~~\text{for all $w\in W^{1,2}_0(\Omega)$}
 \end{equation*}
 
\begin{equation*}
\langle -\Delta u,w\rangle=\int_{\Omega}\nabla u\cdot\nabla w~dx~~~\text{for all $w\in W^{1,2}_0(\Omega)$}.
\end{equation*}
Let us consider $\beta=\frac{u^p-v^p}{u^{p-1}}$, $\eta=\frac{v^p-u^p}{v^{p-1}}$, $\xi=\frac{u^2-v^2}{u}$ and $\zeta=\frac{v^2-u^2}{v}$ as test functions in (\ref{v1}) for any $p>1$. Straightforward computations give,
$$\nabla\left(\frac{u^p-v^p}{u^{p-1}}\right)=\left\{1+(p-1)\left(\frac{v}{u}\right)^p\right\}\nabla u-p\left(\frac{v}{u}\right)^{p-1}\nabla v$$
 $$\nabla\left(\frac{v^p-u^p}{v^{p-1}}\right)=\left\{1+(p-1)\left(\frac{u}{v}\right)^p\right\}\nabla v-p\left(\frac{u}{v}\right)^{p-1}\nabla u$$
 $$\nabla\left(\frac{u^2-v^2}{u}\right)=\left\{1+\left(\frac{v}{u}\right)^2\right\}\nabla u-2\left(\frac{v}{u}\right)\nabla v$$
 $$\nabla\left(\frac{v^2-u^2}{v}\right)=\left\{1+\left(\frac{u}{v}\right)^2\right\}\nabla v-2\left(\frac{u}{v}\right)\nabla u.$$
 Therefore 
 \begin{eqnarray*}
 \langle -\Delta_p u,\frac{u^p-v^p}{u^{p-1}}\rangle&=&\int_{\Omega}\left\{-p\left(\frac{v}{u}\right)^{p-1}|\nabla u|^{p-2}\nabla u\cdot\nabla v+\left(1+(p-1)\left(\frac{v}{u}\right)^p\right)|\nabla u|^p\right\}~dx\\
 &=&\int_{\Omega}\left\{p\left(\frac{v}{u}\right)^{p-1}|\nabla u|^{p-2}\left(|\nabla u||\nabla v|-\nabla u\cdot\nabla v\right)+\left(1+(p-1)\left(\frac{v}{u}\right)^p\right)|\nabla u|^p\right\}~dx\\ &-&\int_{\Omega} p\left(\frac{v}{u}\right)^{p-1}|\nabla u|^{p-1}|\nabla v|~dx
 \end{eqnarray*}
 and 
 \begin{eqnarray*}
 \langle -\Delta u,\frac{u^2-v^2}{u}\rangle=\int_{\Omega}\left\{2\left(\frac{v}{u}\right)\left(|\nabla u||\nabla v|-\nabla u\cdot \nabla v\right)+\left(1+\left(\frac{v}{u}\right)^2\right)|\nabla u|^2-2\left(\frac{v}{u}\right)|\nabla u||\nabla v|\right\}~dx.
 \end{eqnarray*}
 By symmetry we have 
   \begin{eqnarray*}
 \langle -\Delta_p v,\frac{v^p-u^p}{v^{p-1}}\rangle&=&\int_{\Omega}\left\{-p\left(\frac{u}{v}\right)^{p-1}|\nabla v|^{p-2}\nabla v\cdot\nabla u+\left(1+(p-1)\left(\frac{u}{v}\right)^p\right)|\nabla v|^p\right\}~dx\\
 &=&\int_{\Omega}\left\{p\left(\frac{u}{v}\right)^{p-1}|\nabla v|^{p-2}\left(|\nabla v||\nabla u|-\nabla v\cdot\nabla u\right)+\left(1+(p-1)\left(\frac{u}{v}\right)^p\right)|\nabla v|^p\right\}~dx\\ &-& \int_{\Omega}p\left(\frac{u}{v}\right)^{p-1}|\nabla v|^{p-1}|\nabla u|~dx
 \end{eqnarray*}
 and 
 \begin{eqnarray*}
 \langle -\Delta v,\frac{v^2-u^2}{v}\rangle=\int_{\Omega}\left\{2\left(\frac{u}{v}\right)\left(|\nabla v||\nabla u|-\nabla v\cdot \nabla u\right)+\left(1+\left(\frac{u}{v}\right)^2\right)|\nabla v|^2-2\left(\frac{u}{v}\right)|\nabla v||\nabla u|\right\}~dx.
 \end{eqnarray*}
 Thus
 \begin{eqnarray*}
 I(u,v)&=&\int_{\Omega}\left\{p\left(\frac{v}{u}\right)^{p-1}|\nabla u|^{p-2}\left(|\nabla u||\nabla v|-\nabla u\cdot\nabla v\right)+\left(1+(p-1)\left(\frac{v}{u}\right)^p\right)|\nabla u|^p\right\}~dx\\ &-& p\left(\frac{v}{u}\right)^{p-1}|\nabla u|^{p-1}|\nabla v|~dx\\
 &+&\int_{\Omega}\left\{p\left(\frac{u}{v}\right)^{p-1}|\nabla v|^{p-2}\left(|\nabla v||\nabla u|-\nabla v\cdot\nabla u\right)+\left(1+(p-1)\left(\frac{u}{v}\right)^p\right)|\nabla v|^p\right\}~dx\\ &-& p\left(\frac{u}{v}\right)^{p-1}|\nabla v|^{p-1}|\nabla u|~dx\\
 &+&\int_{\Omega}\left\{2\left(\frac{v}{u}\right)\left(|\nabla u||\nabla v|-\nabla u\cdot \nabla v\right)+\left(1+\left(\frac{v}{u}\right)^2\right)|\nabla u|^2-2\left(\frac{v}{u}\right)|\nabla u||\nabla v|\right\}~dx\\
 &+& \int_{\Omega}\left\{2\left(\frac{u}{v}\right)\left(|\nabla v||\nabla u|-\nabla v\cdot \nabla u\right)+\left(1+\left(\frac{u}{v}\right)^2\right)|\nabla v|^2-2\left(\frac{u}{v}\right)|\nabla v||\nabla u|\right\}~dx.
 \end{eqnarray*}
 So
\begin{equation*}
I(u,v)=\int_{\Omega}F(\frac{v}{u},\nabla v, \nabla u)~dx+\int_{\Omega}G(\frac{v}{u},|\nabla v|, |\nabla u|)~dx,
\end{equation*}
 where 
 \begin{eqnarray*}
 F(t,S,R)&=&p\left\{t^{p-1}|R|^{p-2}\left(|R||S|-R\cdot S\right)+t^{1-p}|S|^{p-2}\left(|R||S|-R\cdot S\right)\right\}\\&+&2\left\{t\left(|R||S|-R\cdot S\right)\right\}+ 2\left\{t^{-1}\left(|R||S|-R\cdot S\right)\right\}
\end{eqnarray*} 
and 
\begin{eqnarray*}
 G(t,s,r)&=&\left(1+(p-1)t^p\right)r^p+
 \left(1+(p-1)t^{-p}\right)s^p+(1+t^2)r^2~~~\\&+&(1+t^{-2})s^2-pt^{p-1}r^{p-1}s-pt^{1-p}s^{p-1}r-2trs-2t^{-1}rs,
\end{eqnarray*} 
for all $t=\frac{v}{u}>0, R=\nabla u, S=\nabla v\in \mathbb{R}^N$ and $r=|\nabla u|, s=|\nabla v|\in\mathbb{R}^+.$
We clearly have that $F$ is non-negative. Now let us show that $G$ is non-negative. Indeed, we observe that 
\begin{equation*}
G(t,s,0)=\left(1+(p-1)t^{-p}\right)s^p+(1+t^{-2})s^2\geq 0
\end{equation*}
and 
$G(t,s,0)=0\Rightarrow s=0.$ If $r\ne 0$, by setting $z=\frac{s}{tr}$ we obtain
\begin{eqnarray*}
G(t,s,r)&=&t^pr^p(z^p-pz+(p-1))+r^p((p-1)z^p-pz^{p-1}+1) \nonumber\\ &+&t^2r^2(z^2-2z+1)+r^2(z^2-2z+1),
\end{eqnarray*}
 and $G$ can be written as 
\begin{equation*}
G(t,s,r)=r^p(t^pf(z)+g(z))+r^2(t^2h(z)+k(z)),
\end{equation*}
  with $f(z)=z^p-pz+(p-1),$ $g(z)=(p-1)z^p-pz^{p-1}+1,$ $h(z)=k(z)=z^2-2z+1$ $\forall p>1 .$ We can see that $f,g,h$ and $k$ are non-negative.
  Hence $G$ is non-negative and thus $I(u,v)\geq 0$ for all $(u,v)\in D_I.$ In addition since $f,g,h$ and $k$ vanish if and only if $z=1$, then $G(t,s,r)=0$ if and only if $s=tr.$
Consequently, if $I(u,v)=0$ then we have
\begin{equation*}
\nabla u\cdot \nabla v=|\nabla u||\nabla v|~~\text{and}~~u|\nabla v|=v|\nabla u|
\end{equation*}
  almost everywhere in $\Omega.$ This is equivalent to $\left(u\nabla v-v\nabla u\right)^2=0,$ which implies that $u=\alpha v$ with $\alpha\in\mathbb{R}^{\star}_+.$
\end{proof}
\begin{theorem}
\textup{The first eigenvalues $\lambda$ of equation (\ref{e3}) are simple, i.e. if $u$ and $v$ are two positive first eigenfunctions associated to $\lambda,$ then $u=v.$ }
\end{theorem}
\begin{proof}
By proposition \ref{p1}, we have $u=\alpha v.$ Inserting this into the equation (\ref{e3}) implies that $\alpha=1.$
\end{proof}
\subsection{Invertibility of the operator $-\Delta_p-\Delta$}\label{S5}
To simplify some notations, here we set $X=W^{1,p}_0(\Omega)$ and its dual $X^{\star}=W^{-1,p'}(\Omega),$ where $\frac{1}{p}+\frac{1}{p'}=1.$\\
For the proof of the following lemma, we refer to \cite{PL1}.
\begin{lemma}\label{L3}
\textup{Let $p>2$. Then there exist two positive constants $c_1, c_2$ such that, for all $x_1,x_2\in\mathbb{R}^n,$ we have :}
\begin{enumerate}
\item[(i)] $(x_2-x_1)\cdot (|x_2|^{p-2}x_2-|x_1|^{p-2}x_1)\geq c_1|x_2-x_1|^p$
\item[(ii)]$\left||x_2|^{p-2}x_2-|x_1|^{p-2}x_1\right|\leq c_2(|x_2|+|x_1|)^{p-2}|x_2-x_1|$
\end{enumerate}
\end{lemma}
\begin{proposition}\label{lll1}
\textup{For $p>2,$ the operator $-\Delta_p-\Delta$ is a global homeomorphism.}
\end{proposition}
\noindent The proof is based on the previous Lemma \ref{L3}.
\begin{proof}
Define the nonlinear operator $A:X\rightarrow X^{\star}$ by\\ $\langle Au,v\rangle=\displaystyle{\int_{\Omega}}\nabla u\cdot\nabla v~dx+\displaystyle{\int_{\Omega}}|\nabla u|^{p-2}\nabla u\cdot \nabla v~dx$ for all $u,v\in X.$\\
To show that $-\Delta_p-\Delta$ is a homeomorphism, it is enough to show that $A$ is a continuous strongly monotone operator, [see \cite{KC}, Corollary 2.5.10].\\

\noindent For $p>2,$ for all $u, v\in X$, by $(i)$, we get 
\begin{eqnarray*}
\langle Au-Av,u-v\rangle &=& \int_{\Omega}|\nabla(u-v)|^2dx+\int_{\Omega}\left(|\nabla u|^{p-2}\nabla u-|\nabla v|^{p-2}\nabla v\right)\cdot\nabla(u-v)~dx\\
&\geq& \int_{\Omega}|\nabla(u-v)|^2dx+c_1\int_{\Omega}|\nabla(u-v)|^pdx\\
&\geq&  c_1\|u-v\|^p_{1,p}
\end{eqnarray*}
Thus $A$ is a strongly monotone operator.\\ We claim that $A$ is a continuous operator from $X$ to $X^{\star}.$  Indeed, assume that $u_n\rightarrow u$ in $X.$ We have to show that $\|Au_n-Au\|_{X^{\star}}\rightarrow 0$ as $n\rightarrow\infty.$ Indeed, using $(ii)$ and H\"older's inequality and the Sobolev embedding theorem, one has
\begin{eqnarray*}
\left|\langle Au_n-Au,w\rangle\right| &\leq & \int_{\Omega}\left||\nabla u_n|^{p-2}\nabla u_n-|\nabla u|^{p-2}\nabla u\right||\nabla w|~dx+\int_{\Omega}|\nabla(u_n-u)||\nabla w|~dx\\
&\leq& c_2 \int_{\Omega}\left(|\nabla u_n|+|\nabla u|\right)^{p-2}|\nabla(u_n-u)||\nabla w|~dx+\int_{\Omega}|\nabla(u_n-u)||\nabla w|~dx\\
&\leq& c_2\left(\int_{\Omega}\left(|\nabla u_n|+|\nabla u|\right)^{p}dx\right)^{p-2/p}\left(\int_{\Omega}|\nabla(u_n-u)|^p dx\right)^{1/p}\left(\int_{\Omega}|\nabla w|^p dx\right)^{1/p}\\ &+& c_3\|u_n-u\|_{1,2}\| w\|_{1,2}\\
&\leq & c_4(\|u_n\|_{1,p}+\|u\|_{1,p})^{p-2}\|u_n-u\|_{1,p}\|w\|_{1,p}+c_5\|u_n-u\|_{1,p}\| w\|_{1,p}.
\end{eqnarray*}
Thus $\|Au_n-Au\|_{X^{\star}}\rightarrow 0$, as $n\rightarrow +\infty,$ and hence $A$ is a homeomorphism.
\end{proof}
\section{Bifurcation of eigenvalues}\label{S6}
In the next subsection we show that for equation (\ref{e3}) there is a branch of first eigenvalues bifurcating from $(\lambda_1^D, 0)\in \mathbb{R}^+\times W^{1,p}_0(\Omega).$
\subsection{Bifurcation from zero : the case $p>2$}
By proposition \ref{lll1}, equation (\ref{e3}) is equivalent to
\begin{equation}\label{ss1}
u=\lambda(-\Delta_p-\Delta)^{-1}u~~~\text{for}~~~u\in W^{-1,p'}(\Omega).
\end{equation}
We set \begin{equation}\label{Ss}
S_{\lambda}(u)=u-\lambda(-\Delta_p-\Delta)^{-1}u,
\end{equation}
$~u\in L^2(\Omega)\subset W^{-1,p'}(\Omega)~\text{and}~\lambda>0.$ By $\Sigma=\{(\lambda,u)\in\mathbb{R}^+\times W^{1,p}_0(\Omega)/~u\neq 0~,S_{\lambda}(u)=0 \}$, we denote the set of nontrivial solutions of (\ref{ss1}). 
A bifurcation point for (\ref{ss1}) is a number $\lambda^{\star}\in\mathbb{R}^+$ such that $(\lambda^{\star},0)$ belongs to the closure of $\Sigma.$ This is equivalent to say that, in any neighbourhood of $(\lambda^{\star},0)$ in $\mathbb{R}^+\times W^{1,p}_0(\Omega)$, there exists a nontrivial solution of $S_{\lambda}(u)=0$.\\
Our goal is to apply the Krasnoselski bifurcation theorem [see, \cite{AM}].
\begin{theorem} \textup{(Krasnoselski, 1964)}\\
\textup{Let $X$ be a Banach space and let $T\in C^1(X,X)$ be a compact operator such that $T(0)=0$ and $T'(0)=0.$ Moreover, let $A\in \mathcal{L}(X)$ also be compact. Then every characteristic value $\lambda^{*}$ of $A$ with odd (algebraic) multiplicity is a bifurcation point for $u=\lambda Au+T(u).$}
\end{theorem}
\noindent We now state our bifurcation result.
\begin{theorem}\label{bz}
\textup{Let $p>2.$ Then every eigenvalue $\lambda_k^D$ with odd multiplicity is a bifurcation point in $\mathbb{R}^+\times W^{1,p}_0(\Omega)$ of $S_{\lambda}(u)=0,$ in the sense that in any neighbourhood of $(\lambda^D_k,0)$ in $\mathbb{R}^+\times W^{1,p}_0(\Omega)$ there exists a nontrivial solution of $S_{\lambda}(u)=0$.}
\end{theorem}
\begin{proof}
We write the equation $S_{\lambda}(u)=0$ as $$ u=\lambda Au+T_{\lambda}(u),$$ where $Au=(-\Delta)^{-1}u$ and $T_{\lambda}(u)=[(-\Delta_p-\Delta)^{-1}-(-\Delta)^{-1}](\lambda u)$, where we consider $$(-\Delta_p-\Delta)^{-1}: L^2(\Omega)\subset W^{-1,p'}(\Omega)\rightarrow W^{1,p}_0(\Omega)\subset\subset L^2(\Omega)$$ and $(-\Delta)^{-1} : L^2(\Omega)\subset W^{-1,2}(\Omega)\rightarrow W^{1,2}_0(\Omega)\subset\subset L^2(\Omega).$\\

\noindent For $p>2,$ the mapping $$(-\Delta_p-\Delta)^{-1}-(-\Delta)^{-1}: L^2(\Omega)\subset W^{-1,p'}(\Omega)\rightarrow W^{1,p}_0(\Omega)\subset\subset L^2(\Omega)$$ is compact thanks to Rellich-Kondrachov theorem. We clearly have 
$A\in \mathcal{L}(L^2(\Omega))$ and $T_{\lambda}(0)=0$. Now we have to show that
\begin{enumerate}
\item[(1)]$T_{\lambda}\in C^1$.
\item[(2)] $T_{\lambda}'(0)=0$.
\end{enumerate} 

\noindent In order to show $(1)$ and $(2)$, it suffices to show that 
\begin{enumerate}
\item[(a)] $-\Delta_p-\Delta : W^{1,p}_0(\Omega)\rightarrow W^{-1,p'}(\Omega)$ is continuously differentiable in a neighborhood $u\in W^{1,p}_0(\Omega)$.
\item[(b)] $(-\Delta_p-\Delta)^{-1}$ is a continuous inverse operator.
\end{enumerate}
According to Proposition \ref{lll1}, $-\Delta_p-\Delta$ is a homeomorphism, hence $(-\Delta_p-\Delta)^{-1}$ is continuous and this shows $(b).$ We also recall that in section \ref{S4}, we have shown that $\lambda^D_1$ is simple. \\
\\
Let us show (a). We claim that $-\Delta_p :  W^{1,p}_0(\Omega)\rightarrow W^{-1,p'}(\Omega)$ is G\^ateaux differentiable.
Indeed, for $\varphi\in W^{1,p}_0(\Omega)$ we have
\begin{eqnarray*}
\langle-\Delta_p(u+\delta v), \varphi\rangle-\langle -\Delta_pu,\varphi\rangle &=&\langle|\nabla(u+\delta v)|^{p-2}\nabla(u+\delta v),\nabla\varphi\rangle-\langle|\nabla u|^{p-2}\nabla u,\nabla\varphi\rangle\\
&=& \langle\left(|\nabla(u+\delta v)|^2\right)^{\frac{p-2}{2}}\nabla(u+\delta v),\nabla\varphi\rangle-\langle|\nabla u|^{p-2}\nabla u,\nabla\varphi\rangle\\
&=& \langle\left(|\nabla u|^{2}+2\delta\langle\nabla u, \nabla v\rangle+\delta^2|\nabla v|^2\right)^{\frac{p-2}{2}}\nabla(u+\delta v),\nabla\varphi\rangle\\ &-& \langle|\nabla u|^{p-2}\nabla u,\nabla\varphi\rangle\\
&=&\langle[|\nabla u|^{p-2}+(p-2)|\nabla u|^{2(\frac{p-2}{2}-1)}\delta\langle\nabla u,\nabla v\rangle+O(\delta^2)]\nabla(u+\delta v),\nabla\varphi\rangle \\ &-& \langle|\nabla u|^{p-2}\nabla u,\nabla\varphi\rangle\\
&=&\langle[|\nabla u|^{p-2}+(p-2)|\nabla u|^{p-4}\delta\langle\nabla u,\nabla v\rangle+O(\delta^2)]\nabla(u+\delta v),\nabla\varphi\rangle \\ &-& \langle|\nabla u|^{p-2}\nabla u,\nabla\varphi\rangle\\
&=&(p-2)\delta |\nabla u|^{p-4}\langle\nabla u,\nabla v\rangle\langle\nabla u,\nabla \varphi\rangle+\delta\langle|\nabla u|^{p-2}\nabla v,\nabla\varphi\rangle+O(\delta^2)\\
&=&\delta[(p-2)|\nabla u|^{p-4}\langle\nabla u,\nabla v\rangle\langle\nabla u,\nabla \varphi\rangle+\langle|\nabla u|^{p-2}\nabla v,\nabla\varphi\rangle+O(\delta)].
\end{eqnarray*}
Define $$\langle B(u)v,\varphi\rangle=(p-2)|\nabla u|^{p-4}\langle\nabla u,\nabla v\rangle\langle\nabla u,\nabla \varphi\rangle+\langle|\nabla u|^{p-2}\nabla v,\nabla\varphi\rangle$$ and let $(u_n)_{n\geq0}\subset W^{1,p}_0(\Omega).$ Assume that $u_n\rightarrow u,$ as $n\rightarrow\infty$ in $W^{1,p}_0(\Omega).$ We have
\begin{eqnarray*}
\langle B(u_n)v-B(u)v,\varphi\rangle &=&(p-2)\left[|\nabla u_n|^{p-4}\langle\nabla u_n,\nabla v\rangle\langle\nabla u_n,\nabla \varphi\rangle-|\nabla u|^{p-4}\langle\nabla u,\nabla v\rangle\langle\nabla u,\nabla \varphi\rangle\right]\\
&+& \langle|\nabla u_n|^{p-2}\nabla v,\nabla\varphi\rangle-\langle|\nabla u|^{p-2}\nabla v,\nabla\varphi\rangle.
\end{eqnarray*}
Therefore,
\begin{eqnarray*}
|\langle B(u_n)v-B(u)v,\varphi\rangle| &\leq(p-2) &\left| |\nabla u_n|^{p-4}\langle\nabla u_n,\nabla v\rangle\langle\nabla u_n,\nabla \varphi\rangle-|\nabla u|^{p-4}\langle\nabla u,\nabla v\rangle\langle\nabla u,\nabla \varphi\rangle\right|\\
&+& \left||\nabla u_n|^{p-2}-|\nabla u|^{p-2}\right||\langle\nabla v,\nabla\varphi\rangle|.
\end{eqnarray*}
By assumption, we can assume that, up to subsequences,
\begin{enumerate}
\item[$(\ast)$] $\nabla u_n\rightarrow\nabla u$ in $\left(L^p(\Omega)\right)^N$ as $n\rightarrow\infty$ and
\item[$(\ast\ast)$]$\nabla u_n(x)\rightarrow\nabla u(x)$ almost everywhere as $n\rightarrow\infty.$
\end{enumerate}
Then $ |\nabla u_n|^{p-4}\langle\nabla u_n,\nabla v\rangle\langle\nabla u_n,\nabla \varphi\rangle\rightarrow |\nabla u|^{p-4}\langle\nabla u,\nabla v\rangle\langle\nabla u,\nabla \varphi\rangle $ and $|\nabla u_n|^{p-2}\rightarrow |\nabla u|^{p-2},$ as $n\rightarrow\infty.$ Consequently $\langle B(u_n)v,\varphi\rangle\rightarrow \langle B(u )v,\varphi\rangle$as $n\rightarrow\infty.$
Thus, we find that $-\Delta_p-\Delta\in C^1$ and thanks to the Inverse function theorem $(-\Delta_p-\Delta)^{-1}$ is differentiable in a neighborhood of $u\in W^{1,p}_0(\Omega)$. Therefore according to the Krasnoselski bifurcation Theorem, we obtain that $\lambda_k^D$ is a bifurcation point at zero.
\end{proof}
\subsection{Bifurcation from infinity : the case $1<p<2$}
\noindent We recall the nonlinear eigenvalue problem we are investigating,
\begin{equation}\label{em1}
\left\{
\begin{array}{l}
-\Delta_p u-\Delta u=\displaystyle \lambda u~~\text{in $\Omega$},\\
u = \displaystyle 0~~~~~~~~~~~~~~~~\text{on $\partial\Omega$}
\end{array}
\right.
\end{equation}
Under a solution of (\ref{em1}) (for $1<p<2$), we understand a pair $(\lambda,u)\in\mathbb{R}^+_{\star}\times W^{1,2}_0(\Omega)$ satisfying the integral equality,
\begin{equation}\label{em2}
\int_{\Omega}|\nabla u|^{p-2}\nabla u\cdot\nabla\varphi~dx+\int_{\Omega}\nabla u\cdot\nabla\varphi~dx=\lambda\int_{\Omega}u\varphi~dx~~\text{for every $\varphi\in W^{1,2}_0(\Omega)$}.
\end{equation}
\begin{definition}
\textup{Let $\lambda\in\mathbb{R}.$ We say that the pair $(\lambda,\infty)$ is a bifurcation point from infinity for problem (\ref{em1}) if there exists a sequence of pairs $\{(\lambda_n,u_n)\}_{n=1}^{\infty}\subset\mathbb{R}\times W^{1,p}_0(\Omega)$ such that equation (\ref{em2}) holds and $(\lambda_n,\|u_n\|_{1,2})\rightarrow(\lambda,\infty).$}
\end{definition}
\noindent We now state  the main theorem.
\begin{theorem}\label{tt1}
\textup{The pair $(\lambda^D_1,\infty)$ is a bifurcation point from infinity for the problem (\ref{em1}).}
\end{theorem}
\noindent For $u\in W^{1,2}_0(\Omega),~u\neq 0,$ we set $v=u/\|u\|_{1,2}^{2-\frac{1}{2}p}$.  We have $\|v\|_{1,2}=\frac{1}{\|u\|_{1,2}^{1-\frac{1}{2}p}}$ and 
$$|\nabla v|^{p-2}\nabla v=\frac{1}{\|u\|_{1,2}^{(2-\frac{1}{2}p)(p-1)}} |\nabla u|^{p-2}\nabla u.$$
Introducing this change of variable in (\ref{em2}), we find that, 
\begin{equation}\label{em3}
\|u\|_{1,2}^{(2-\frac{1}{2}p)(p-2)}\int_{\Omega}|\nabla v|^{p-2}\nabla v\cdot\nabla\varphi~dx+\int_{\Omega}\nabla v\cdot\nabla\varphi~dx=\lambda\int_{\Omega}v\varphi~dx~~\text{for every $\varphi\in W^{1,2}_0(\Omega)$}.
\end{equation}
But, on the other hand, we have
$$\|v\|^{p-4}_{1,2}=\frac{1}{\|u\|_{1,2}^{(1-\frac{1}{2}p)(p-4)}}=\frac{1}{\|u\|_{1,2}^{(2-\frac{1}{2}p)(p-2)}}.$$ Consequently it follows that equation (\ref{em3}) is equivalent to
\begin{equation}\label{em4}
\|v\|_{1,2}^{4-p}\int_{\Omega}|\nabla v|^{p-2}\nabla v\cdot\nabla\varphi~dx+\int_{\Omega}\nabla v\cdot\nabla\varphi~dx=\lambda\int_{\Omega}v\varphi~dx~~\text{for every $\varphi\in W^{1,2}_0(\Omega)$}.
\end{equation}
This leads to the following nonlinear eigenvalue problem (for $1<p<2$)
\begin{equation}\label{em5}
\left\{
\begin{array}{l}
-\|v\|_{1,2}^{4-p}\Delta_p v-\Delta v =\lambda v~~\text{in $\Omega$},\\
v = \displaystyle 0~~~~~~~~~~~~~~~~\text{on $\partial\Omega$}
\end{array}
\right.
\end{equation}
The proof of Theorem \ref{tt1} follows immediately from the following remark, and the proof that $(\lambda_1^D,0)$ is a bifurcation of (\ref{em5}).
\begin{remark}\label{RMK}
\textup{With this transformation, we have that the pair $(\lambda^D_1,\infty)$ is a bifurcation point for the problem (\ref{em1}) if and only if the pair $(\lambda^D_1,0)$ is a bifurcation point for the problem (\ref{em5}).}
\end{remark}
\noindent Let us consider a small ball $B_r(0) :=\{~w~\in W^{1,2}_0(\Omega)/~~~\|w\|_{1,2}< r~\},$ and
consider the operator $$T :=-\|\cdot\|_{1,2}^{4-p}\Delta_p-\Delta : W^{1,2}_0(\Omega)\rightarrow W^{-1,2}(\Omega)$$
\begin{proposition}
\textup{Let $1<p<2.$ There exists $r>0$ such that the mapping \\$T : B_r(0)\subset W^{1,2}_0(\Omega)\rightarrow W^{-1,2}(\Omega)$ is invertible, with a continuous inverse.}
\end{proposition}
\begin{proof}
In order to prove that the operator $T$ is invertible with a continuous inverse, we again rely on  [ \cite{KC}, Corollary 2.5.10]. 
We show that there exists $\delta>0$ such that $$\langle T(u)-T(v), u-v\rangle\geq \delta\|u-v\|^2_{1,2}, ~~\text{for}~~u,v\in B_r(0)\subset W^{1,2}_0(\Omega)$$ with $r>0$ sufficiently small.\\
Indeed, using that $-\Delta_p$ is strongly monotone on $W^{1,p}_0(\Omega)$ on the one hand and the H\"older inequality on the other hand, we have
\begin{eqnarray}\label{bf1}\notag
\langle T(u)-T(v), u-v\rangle&=&\|\nabla u-\nabla v\|^2_{2}+\left(\|u\|_{1,2}^{4-p}(-\Delta_pu)-\|v\|_{1,2}^{4-p}(-\Delta_pv), u-v \right)\\ \notag
&=&\|u-v\|^2_{1,2}+\|u\|_{1,2}^{4-p}\left((-\Delta_pu)-(-\Delta_p v), u-v\right)\\ \notag
&+& \left(\|u\|_{1,2}^{4-p}-\|v\|_{1,2}^{4-p} \right)\left(-\Delta_pv, u-v\right)\\ \notag
&\geq & \|u-v\|^2_{1,2}-\left|\|u\|_{1,2}^{4-p}-\|v\|_{1,2}^{4-p}\right|\|\nabla v\|_{p}^{p-1}\|\nabla (u-v)\|_{p}\\
&\geq& \|u-v\|^2_{1,2}-\left|\|u\|_{1,2}^{4-p}-\|v\|_{1,2}^{4-p}\right|C\| v\|_{1,2}^{p-1}\| u-v\|_{1,2}.
\end{eqnarray}
Now, we obtain by the Mean Value Theorem  that there exists $\theta\in [0,1]$ such that
\begin{eqnarray*}
\left|\|u\|_{1,2}^{4-p}-\|v\|_{1,2}^{4-p}\right|&=&\left|\frac{d}{dt}\left(\|u+t(v-u)\|^2_{1,2}\right)^{2-\frac{1}{2}p}|_{t=\theta} (v-u)\right|\\
&=&\left|(2-\frac{1}{2}p)\left(\|u+\theta(v-u)\|^2_{1,2}\right)^{1-\frac{1}{2}p}2\left(u+\theta(v-u),v-u\right)_{1,2}\right|\\
&\leq& (4-p)\|u+\theta(v-u)\|^{2-p}_{1,2}\|u+\theta(v-u)\|_{1,2}\|u-v\|_{1,2}\\
&=&(4-p)\|u+\theta(v-u)\|_{1,2}^{3-p}\|u-v\|_{1,2}\\
&\leq & (4-p)\left((1-\theta)\|u\|_{1,2}+\theta\|v\|_{1,2}\right)^{3-p}\|u-v\|_{1,2}\\
&\leq & (4-p) r^{3-p}\|u-v\|_{1,2}.
\end{eqnarray*}
Hence, continuing with the estimate of equation (\ref{bf1}), we get
$$\langle T(u)-T(v), u-v\rangle\geq \|u-v\|^2_{1,2}(1-(4-p) r^{3-p}Cr^{p-1})=\|u-v\|^2_{1,2}(1-C'r^{2}),$$ and thus the claim, for $r>0$ small enough.\\ Hence, the operator $T$ is strongly monotone on $B_r(0)$ and it is continuous, and hence the claim follows.
\end{proof}
\noindent Clearly the mappings
$$T_{\tau}=-\Delta-\tau\|\cdot\|^{\gamma}_{1,2}\Delta_p : B_r(0)\subset W^{1,2}_0(\Omega)\rightarrow W^{-1,2}(\Omega),~~~0\leq \tau\leq 1 $$ are also local homeomorphisms for $1<p<2$ with $\gamma=4-p>0$. Consider now the homotopy maps $$H(\tau, y):=(-\tau\|\cdot\|_{1,2}^{\gamma}\Delta_p-\Delta)^{-1}(y),~~~y\in T_{\tau}(B_r(0))\subset W^{-1,2}(\Omega).$$ Then we can find a $\rho>0$ such that the ball $$B_{\rho}(0)\subset \bigcap_{0\leq \tau\leq 1} T_{\tau}(B_r(0))$$ and $$H(\tau,\cdot) : B_{\rho}(0)\cap L^2(\Omega)\mapsto W^{1,2}_0(\Omega)\subset\subset L^2(\Omega)$$ are compact mappings.
Set now $$\tilde{S}_{\lambda}(u)=u-\lambda (-\|u\|_{1,2}^{\gamma}\Delta_p-\Delta)^{-1}u.$$ Notice that $\tilde{S}_{\lambda}$ is a compact perturbation of the identity in $L^2(\Omega).$
We have $0\notin H([0,1]\times\partial B_r(0) ).$ So it makes sense to consider the Leray-Schauder topological degree of $H(\tau,\cdot)$ on $B_r(0).$ And by the property of the invariance by homotopy, one has
\begin{equation}\label{eqa1}
\deg(H(0,\cdot),B_r(0),0)=\deg(H(1,\cdot),B_r(0),0).
\end{equation}
\begin{theorem}\label{tt}
\textup{The pair $(\lambda^D_1,0)$ is a bifurcation point in $\mathbb{R}^+\times L^2(\Omega)$ of $\tilde{S}_{\lambda}(u)=0$, for $1<p<2.$}
\end{theorem}
\begin{proof}
Suppose by contradiction that $(\lambda^D_1,0)$ is not a bifurcation for $\tilde{S}_{\lambda}.$ Then, there exist $\delta_0>0$ such that for all $r\in (0,\delta_0)$ and $\varepsilon\in (0,\delta_0),$
\begin{equation}\label{eqaa2}
\tilde{S}_{\lambda}(u)\neq 0~~~\forall~|\lambda^D_1-\lambda|\leq \varepsilon,~\forall~u\in L^2(\Omega),~\|u\|_{2}=r.
\end{equation}
Taking into account that (\ref{eqaa2}) holds, it follows that it make sense to consider the Leray-Schauder topological degree $\deg(\tilde{S}_{\lambda}, B_r(0), 0)$ of $\tilde{S}_{\lambda}$ on $B_r(0).$\\
\\
We observe that 
\begin{equation}\label{zoo1}
\left(I-(\lambda_1^D-\varepsilon)H(\tau,\cdot)\right)|_{\partial B_r(0)}\neq 0~~\text{for}~~ \tau\in [0,1].
\end{equation}

\noindent Proving (\ref{zoo1}) garantee the well posedness of $\deg(I-(\lambda^D_1\pm\varepsilon)H(\tau,\cdot), B_r(0), 0)$ for any $\tau\in [0,1].$
\noindent Indeed, by contradiction suppose that there exists $v\in \partial B_r(0)\subset L^2(\Omega)$ such that \\$v-(\lambda_1^D-\varepsilon)H(\tau,v)=0,$ for some $\tau\in [0,1].$\\
One concludes that then $v\in W^{1,2}_0(\Omega),$ and then that
$$-\Delta v-\tau\|v\|^{\gamma}_{1,2}\Delta_pv=(\lambda^D_1-\varepsilon)v.$$ However, we get the contradiction,
$$(\lambda^D_1-\varepsilon)\|v\|_2^2=\|\nabla v\|^2_2+\tau\| v\|^{\gamma}_{1,2}\|\nabla v\|^p_p\geq \|\nabla v\|^2_2\geq \lambda^D_1\|v\|^2_2.$$
By the contradiction assumption, we have
\begin{equation}\label{eqaa3}
\deg(I-(\lambda^D_1+\varepsilon)H(1,\cdot), B_r(0), 0)=\deg(I-(\lambda^D_1-\varepsilon)H(1,\cdot), B_r(0), 0).
\end{equation}
By homotopy using (\ref{eqa1}), we have
\begin{eqnarray}\label{c1a}
\deg(I-(\lambda^D_1-\varepsilon)H(1,\cdot), B_r(0), 0)&=&\deg(I-(\lambda^D_1-\varepsilon)H(0,\cdot), B_r(0), 0)\\
&=&\deg(I-(\lambda^D_1-\varepsilon)(-\Delta)^{-1}, B_r(0), 0)=1\notag
\end{eqnarray}
Now, using (\ref{c1a}) and (\ref{eqaa3}), we find that
\begin{equation}\label{cc2a}
\deg(I-(\lambda^D_1+\varepsilon)H(1,\cdot), B_r(0), 0)=\deg(I-(\lambda^D_1-\varepsilon)H(0,\cdot), B_r(0), 0)=1
\end{equation}

\noindent Furthermore, since $\lambda^D_1$ is a simple eigenvalue of $-\Delta,$ it is well-known [see \cite{AM}] that
\begin{equation}\label{zo2}
\deg(I-(\lambda^D_1+\varepsilon)(-\Delta)^{-1}, B_r(0), 0)=\deg(I-(\lambda^D_1+\varepsilon)H(0,\cdot), B_r(0), 0)=-1
\end{equation}
In order to get contradiction (to relation  (\ref{cc2a})), it is enough to show that, 
\\
\begin{equation}\label{c3}
\deg(I-(\lambda^D_1+\varepsilon)H(1,\cdot), B_r(0), 0)=\deg(I-(\lambda^D_1+\varepsilon)H(0,\cdot), B_r(0), 0),
\end{equation}
$r>0$ sufficiently small.
We have to show that $$\left(I-(\lambda_1^D+\varepsilon)H(\tau,\cdot)\right)|_{\partial B_r(0)}\neq 0~~~\text{for}~~\tau\in[0,1].$$ 
Suppose by contradiction that there is $r_n\rightarrow0$, $\tau_n\in [0,1]$ and $u_n\in \partial B_{r_n}(0)$ such that
$$u_n-(\lambda_1^D+\varepsilon)H(\tau_n,u_n)=0$$ or equivalently
\begin{equation}\label{zoz1}
-\tau_n\|u_n\|^{\gamma}_{1,2}\Delta_p u_n-\Delta u_n=(\lambda_1^D+\varepsilon)u_n.
\end{equation}
Dividing the equation (\ref{zoz1}) by $\|u_n\|_{1,2}$, we obtain
$$-\tau_n\|u_n\|^{\gamma+p-1}_{1,2}\Delta_p\left(\frac{u_n}{\|u_n\|_{1,2}}\right)-\Delta\left(\frac{u_n}{\|u_n\|_{1,2}}\right)=(\lambda_1^D+\varepsilon)\frac{u_n}{\|u_n\|_{1,2}},$$ and by setting $v_n=\frac{u_n}{\|u_n\|_{1,2}},$ it follows that
\begin{equation}\label{zoz2}
-\tau_n\|u_n\|^{\gamma+p-1}_{1,2}\Delta_p v_n-\Delta v_n=(\lambda_1^D+\varepsilon)v_n.
\end{equation}
But since $\|v_n\|_{1,2}=1,$ we have $v_n\rightharpoonup v$ in $W^{1,2}_0(\Omega)$ and $v_n\rightarrow v$ in $L^2(\Omega).$ Furthermore, the first term in the left hand side of equation (\ref{zoz2}) tends to zero in $W^{-1,p'}(\Omega)$ as $r_n\rightarrow 0$ and hence in $W^{-1,2}(\Omega).$ Equation (\ref{zoz1}) then implies that $v_n\rightarrow v$  strongly in $W^{1,2}_0(\Omega)$ since $-\Delta: W^{1,2}_0(\Omega)\rightarrow 	W^{-1,2}(\Omega)$ is a homeomorphism and thus $v$ with $\|v\|_{1,2}=1$ solves $-\Delta v=(\lambda_1^D+\varepsilon)v$, which is impossible because $\lambda_1^D+\varepsilon$ is not the first eigenvalue of $-\Delta$ on $W^{1,2}_0(\Omega)$ for $\varepsilon>0.$\\
\\
Therefore, by homotopy it follows that
$$\deg(I-(\lambda^D_1+\varepsilon)H(1,\cdot), B_r(0), 0)=\deg(I-(\lambda^D_1+\varepsilon)H(0,\cdot), B_r(0), 0).$$ Now, thanks to (\ref{zo2}), we find that $$\deg(I-(\lambda^D_1+\varepsilon)H(1,\cdot), B_r(0), 0)=-1,$$  which contradicts equation (\ref{cc2a}).
\end{proof}
\begin{theorem}\label{tt}
\textup{The pair $(\lambda^D_k,0)$ ($k>1$) is a bifurcation point of $\tilde{S}_{\lambda}(u)=0$, for $1<p<2$ if $\lambda^D_k$ is of odd multiplicity.}
\end{theorem}
\begin{proof}
Suppose by contradiction that $(\lambda^D_k,0)$ is not a bifurcation for $\tilde{S}_{\lambda}.$ Then, there exist $\delta_0>0$ such that for all $r\in (0,\delta_0)$ and $\varepsilon\in (0,\delta_0),$
\begin{equation}\label{eqa2}
\tilde{S}_{\lambda}(u)\neq 0~~~\forall~|\lambda^D_k-\lambda|\leq \varepsilon,~\forall~u\in L^2(\Omega),~\|u\|_{2}=r.
\end{equation}
Taking into account that (\ref{eqa2}) holds, it follows that it make sense to consider the Leray-Schauder topological degree $\deg(\tilde{S}_{\lambda}, B_r(0), 0)$ of $\tilde{S}_{\lambda}$ on $B_r(0).$\\
We show that 
\begin{equation}\label{zo1}
\left(I-(\lambda_k^D-\varepsilon)H(\tau,\cdot)\right)|_{\partial B_r(0)}\neq 0~~\text{for}~~ \tau\in [0,1].
\end{equation}
Proving (\ref{zo1}) garantee the well posedness of $\deg(I-(\lambda^D_k\pm\varepsilon)H(\tau,\cdot), B_r(0), 0)$ for any $\tau\in [0,1].$ Indeed, consider the projections $P^-$ and $P^+$ onto the spaces $\text{span}\{e_1,\dots,e_{k-1}\}$ and $\text{span}\{e_k,e_{k+1},\dots\} $, respectively, where $e_1\dots,e_k,e_{k+1},\dots$denote the eigenfunctions associated to the Dirichlet problem (\ref{e1}).\\
Suppose by contradiction that relation (\ref{zo1}) does not hold. Then there exists $v\in \partial B_r(0)\subset L^2(\Omega)$ such that $v-(\lambda^D_k-\varepsilon)H(\tau,v)=0,$ for some $\tau\in [0,1].$ This is equivalent of having 
\begin{equation}\label{ber1}
-\Delta v-(\lambda^D_k-\varepsilon)v=\tau\|v\|^{\gamma}_{1,2}\Delta_p v.
\end{equation}

Replacing $v$ by $P^+v+P^-v,$ and multiplying equation (\ref{ber1}) by $P^+v-P^-v$ in the both sides, we obtain
$$\langle [-\Delta-(\lambda^D_k-\varepsilon)](P^+v+P^-v),P^+v-P^-v\rangle=\tau\|P^+v+P^-v\|^{\gamma}_{1,2}\langle\Delta_p[P^+v+P^-v], P^+v-P^-v\rangle$$
$$\Updownarrow$$
\begin{eqnarray*}
-\left[\|\nabla P^-v\|^2_2-(\lambda^D_k-\varepsilon)\|P^-v\|^2_2\right]+\|\nabla P^+v\|^2_2-(\lambda^D_k-\varepsilon)\|P^+v\|^2_2 &= &\tau\|P^+v+P^-v\|^{\gamma}_{1,2}\\ 
&\times & \langle\Delta_p[P^+v+P^-v], P^+v-P^-v\rangle.
\end{eqnarray*}
But $$\langle\Delta_p[P^+v+P^-v], P^+v-P^-v\rangle=-\int_{\Omega}|\nabla (P^+v+P^-v)|^{p-2}\nabla (P^+v+P^-v)\cdot \nabla (P^+v-P^-v)~dx,$$ and using the H\"older inequality, the embedding $W^{1,2}_0(\Omega)\subset W^{1,p}_0(\Omega)$ and the fact that $P^+v$ and $P^-v$ don't vanish simultaneously, there is some positive constant $C'>0$ such that \\
$\|P^+v-P^-v\|_{1,2}\leq C'(\|P^+v\|^2_{1,2}+\|P^-v\|^2_{1,2})=C'\|P^+v-P^-v\|^2_{1,2},~~\text{since}~\left(P^+v,P^-v\right)_{1,2}=0,$  we have
\begin{eqnarray*}
\left|\langle\Delta_p[P^+v+P^-v], P^+v-P^-v\rangle\right | &\leq & \|P^+v+P^-v\|^{p-1}_{1,p}\|P^+v-P^-v\|_{1,p}\\
&\leq & C'\|P^+v+P^-v\|^{p-1}_{1,2}\|P^+v-P^-v\|^{2}_{1,2}\\
&\leq& C'\|P^+v+P^-v\|^{p+1}_{1,2},~~\text{since}~~\|P^+v-P^-v\|^{2}_{1,2}=\|P^+v+P^-v\|^{2}_{1,2}.
\end{eqnarray*}

On the other hand, thanks to the Poincaré inequality as well as the variational characterization of eigenvalues we find $$-\left[\|\nabla P^-v\|^2_2-(\lambda^D_k-\varepsilon)\|P^-v\|^2_2\right]\geq 0$$ and $$\|\nabla P^+v\|^2_2-(\lambda^D_k-\varepsilon)\|P^+v\|^2_2\geq 0,$$  we can bound from below these two inequalities together by $\|\nabla P^+v\|^2_{2}+\|\nabla P^-v\|^2_{2}.$ \\
Finally, we have 
$$\|v\|^2_{1,2}=\|\nabla P^+v\|^2_{2}+\|\nabla P^-v\|^2_{2}\leq \tau C'\|P^+v+P^-v\|^{\gamma+p+1}_{1,2},~~\text{with}~~\gamma=4-p,$$
$$\Updownarrow$$
$$\|v\|^2_{1,2}\leq C''\|v\|^{\gamma+p+1}_{1,2}\Leftrightarrow 1\leq C'' r^{3}\rightarrow 0,$$ for $r$ taken small enough. This shows that (\ref{zo1}) holds.\\
\noindent By the contradiction assumption, we have
\begin{equation}\label{eqa3}
\deg(I-(\lambda^D_k+\varepsilon)H(1,\cdot), B_r(0), 0)=\deg(I-(\lambda^D_k-\varepsilon)H(1,\cdot), B_r(0), 0).
\end{equation}
By homotopy using (\ref{zo1}), we have
\begin{eqnarray}\label{c1}
\deg(I-(\lambda^D_k-\varepsilon)H(1,\cdot), B_r(0), 0)&=&\deg(I-(\lambda^D_k-\varepsilon)H(0,\cdot), B_r(0), 0)\\
&=&\deg(I-(\lambda^D_k-\varepsilon)(-\Delta)^{-1}, B_r(0), 0)=(-1)^{\beta}\notag,
\end{eqnarray}
where $\beta$ is the sum of algebraic multiplicities of the  eigenvalues $\lambda_k^D-\varepsilon< \lambda.$
Similarly, if $\beta'$ denotes the sum of the algebraic multiplicities of the characteristic values of $(-\Delta)^{-1}$ such that $\lambda>\lambda^D_k+\varepsilon,$ then 
\begin{equation}\label{cpl}
\deg(I-(\lambda^D_k+\varepsilon)H(1,\cdot), B_r(0), 0)=(-1)^{\beta'}
\end{equation}
But since $[\lambda^D_k-\varepsilon,\lambda^D_k+\varepsilon]$ contains only the eigenvalue $\lambda^D_k,$ it follows that $\beta'=\beta+\alpha,$ where $\alpha$ denotes the algebraic multiplicity of $\lambda^D_k.$ Consequently, we have
\begin{eqnarray*}
\deg(I-(\lambda^D_k+\varepsilon)H(1,\cdot), B_r(0), 0)&=&(-1)^{\beta+\alpha}\\
&=&(-1)^{\alpha}\deg(I-(\lambda^D_k+\varepsilon)H(1,\cdot), B_r(0), 0)\\
&=& -\deg(I-(\lambda^D_k+\varepsilon)H(1,\cdot), B_r(0), 0),
\end{eqnarray*}
since $\lambda^D_k$ is with odd multiplicity. This contradicts (\ref{eqa3}).
\end{proof}
\section{Multiple solutions}\label{S}
\noindent In this section we prove multiciplity results by distinguishing again the two cases $1<p<2$ and $p>2$. We recall the following definition which will be used in this section. Let $X$ be a Banach space and $\Omega\subset X$ an open bounded domain which is symmetric with respect to the origin of $X,$ that is, $u\in\Omega\Rightarrow -u\in\Omega.$ Let $\Gamma$ be the class of all the symmetric subsets $A\subseteq X\backslash\{0\}$ which are closed in $X\backslash\{0\}.$
\begin{definition}\textup{(Krasnoselski genus)}\\
\textup{Let $A\in\Gamma$. The genus of $A$ is the least integer $p\in\mathbb{N}^*$ such that there exists $\Phi: A\rightarrow\mathbb{R}^p$ continuous, odd and such that $\Phi(x)\neq 0$ for all $x\in A.$ The genus of $A$ is usually denoted by $\gamma(A)$.}
\end{definition}
\begin{theorem}
\textup{Let $1<p<2$ or $2<p<\infty$, and suppose that $\lambda\in (\lambda^D_k,\lambda^D_{k+1})$ for any $k\in\mathbb{N^*}.$ Then equation (\ref{e3}) has at least $k$ pairs of nontrivial solutions.}
\end{theorem}
\begin{proof}
\noindent Case 1: $1<p<2$.\\
In this case we will avail of [\cite{AM}, Proposition 10.8].
We consider the energy functional $I_{\lambda}: W^{1,2}_0(\Omega)\backslash\{0\}\rightarrow\mathbb{R}$ associated to the problem (\ref{e3}) defined by
$$I_{\lambda}(u)=\frac{2}{p}\int_{\Omega}|\nabla u|^p~dx+\int_{\Omega}|\nabla u|^2~dx-\lambda\int_{\Omega}u^2~dx.$$
The functional $I_{\lambda}$ is not bounded from below on $W^{1,2}_0(\Omega)$, so we consider again the natural constraint set, the Nehari manifold on which we minimize the functional $I_{\lambda}.$ The Nehari manifold is given by
$$\mathcal{N}_{\lambda}:=\{u\in W^{1,2}_0(\Omega)\backslash\{0\}:~\langle I'_{\lambda}(u),u\rangle=0\}.$$
On $\mathcal{N}_{\lambda},$ we have $I_{\lambda}(u)=(\frac{2}{p}-1)\displaystyle{\int_{\Omega}}|\nabla u|^p~dx>0.$ We clearly have that,
 $I_{\lambda}$ is even and bounded from below on $\mathcal{N}_{\lambda}.$\\ Now, let us show that every (PS) sequence for $I_{\lambda}$ has a converging subsequence on $\mathcal{N}_{\lambda}.$ Let $(u_n)_n$ be a $(PS)$ sequence, i.e, $|I_{\lambda}(u_n)|\leq C$, for all $n$, for some $C>0$ and $I'_{\lambda}(u_n)\rightarrow 0$ in $W^{-1,2}(\Omega)$ as $n\rightarrow +\infty.$  \\
We first show that the sequence $(u_n)_n$ is bounded on $\mathcal{N}_{\lambda}$.
Suppose by contradiction that this is not true, so $\displaystyle{\int_{\Omega}}|\nabla u_n|^2~dx\rightarrow +\infty$ as $n\rightarrow +\infty.$ Since $I_{\lambda}(u_n)=(\frac{2}{p}-1)\displaystyle{\int}_{\Omega}|\nabla u_n|^p~dx$ we have $\displaystyle{\int}_{\Omega}|\nabla u_n|^p~dx\leq c.$ On $\mathcal{N}_{\lambda}$, we have
 \begin{equation}\label{m1}
 0<\int_{\Omega}|\nabla u_n|^p~dx=\lambda\int_{\Omega}u_n^2~dx-\int_{\Omega}|\nabla u_n|^2~dx,
 \end{equation}
 and hence $\displaystyle{\int_{\Omega}}u_n^2~dx\rightarrow+\infty.$ Let $v_n=\frac{u_n}{\|u_n\|_2}$ then  $\displaystyle{\int_{\Omega}}|\nabla v_n|^2~dx\leq \lambda$ and hence $v_n$ is bounded in $W^{1,2}_0(\Omega).$ Therefore there exists $v_0\in W^{1,2}_0(\Omega)$ such that $v_n\rightharpoonup v_0$ in $W^{1,2}_0(\Omega)$ and $v_n\rightarrow v_0$ in $L^2(\Omega).$ Dividing (\ref{m1}) by $\|u_n\|^p_2,$ we have
 $$\frac{\lambda\displaystyle{\int_{\Omega}}u_n^2~dx-\int_{\Omega}|\nabla u_n|^2~dx}{\|u_n\|^p_2}=\int_{\Omega}|\nabla v_n|^p~dx\rightarrow 0,$$ since $\lambda\displaystyle{\int_{\Omega}}u_n^2~dx-\int_{\Omega}|\nabla u_n|^2~dx=(\frac{2}{p}-1)^{-1}I_{\lambda}(u_n)$, $|I_{\lambda}(u_n)|\leq C$ and $\|u_n\|^p_2\rightarrow +\infty.$ Now, since $v_n\rightharpoonup v_0$ in $W^{1,2}_0(\Omega)\subset W^{1,p}_0(\Omega),$ we infer that $$\int_{\Omega}|\nabla v_0|^p~dx\leq \liminf_{n\rightarrow +\infty}\int_{\Omega}|\nabla v_n|^p~dx=0,$$ and consequently $v_0=0.$ So $v_n\rightarrow 0$ in $L^2(\Omega)$ and this is a contradiction since $\|v_n\|_2=1.$ So $(u_n)_n$ is bounded on $\mathcal{N}_{\lambda}$.\\ Next, we show that $u_n$ converges strongly to $u$ in $W^{1,2}_0(\Omega).$\\ To do this, we will use the following vectors inequality for $1<p<2$
 \begin{equation*}
 (|x_2|^{p-2}x_2-|x_1|^{p-2}x_1)\cdot(x_2-x_1)\geq C'(|x_2|+|x_1|)^{p-2}|x_2-x_1|^2,
 \end{equation*}
 for all $x_1,x_2\in\mathbb{R}^N$ and for some $C'>0, $ [see~\cite{PL1}].\\
 We have $\displaystyle{\int_{\Omega}}u_n^2~dx\rightarrow\displaystyle{\int_{\Omega}}u^2~dx$ and since $I_{\lambda}'(u_n)\rightarrow 0$ in $W^{-1,2}(\Omega),$ $u_n\rightharpoonup u$ in $W^{1,2}_0(\Omega),$ we also have $I_{\lambda}'(u_n)(u_n-u)\rightarrow 0$ and $I_{\lambda}'(u)(u_n-u)\rightarrow 0$ as $n\rightarrow +\infty.$ On the other hand, one has
 \begin{eqnarray*}
 \langle I'_{\lambda}(u_n)- I'_{\lambda}(u), u_n-u\rangle &=& 2\left[\int_{\Omega}\left(|\nabla u_n|^{p-2}\nabla u_n-|\nabla u|^{p-2}\nabla u\right)\cdot \nabla(u_n-u)~dx\right]\\
 &+& 2\int_{\Omega}|\nabla(u_n-u)|^2~dx-2\lambda\int_{\Omega}|u_n-u|^2~dx\\
 &\geq & C'\int_{\Omega}\left(|\nabla u_n|+|\nabla u|\right)^{p-2}|\nabla (u_n-u)|^2~dx\\
 &+& 2\int_{\Omega}|\nabla(u_n-u)|^2~dx-2\lambda\int_{\Omega}|u_n-u|^2~dx\\
 &\geq&  2\int_{\Omega}|\nabla(u_n-u)|^2~dx-2\lambda\int_{\Omega}|u_n-u|^2~dx\\
 &\geq & \|u_n-u\|^2_{1,2}-\lambda\int_{\Omega}|u_n-u|^2~dx.
 \end{eqnarray*}
 Therefore $\|u_n-u\|_{1,2}\rightarrow 0$ as $n\rightarrow +\infty$ and $u_n$ converges strongly to $u$ in $W^{1,2}_0(\Omega).$\\
 \\
 Let $\Sigma'=\{A\subset\mathcal{N}_{\lambda}:~A~\text{closed}~\text{and}~-A=A\}$ and $\Gamma_j=\{A\in\Sigma':~\gamma(A)\geq j\},$ where $\gamma(A)$ denotes the Krasnoselski's genus. We show that $\Gamma_j\neq\emptyset.$
 \\
Set $E_j=\text{span}\{e_i,~~i=1,\dots,j\},$ where $e_i$ are the eigenfunctions associated to the problem (\ref{e1}). Let $\lambda\in (\lambda^D_j,\lambda^D_{j+1}),$ and consider $v\in S_{j}:=\{ v\in E_j:~\int_{\Omega}|v|^2~dx=1\}.$ Then set
$$\rho(v)=\left[\frac{\int_{\Omega}|\nabla v|^p~dx}{\lambda\int_{\Omega}v^2~dx-\int_{\Omega}|\nabla v|^2~dx}\right]^{\frac{1}{2-p}}.$$ Then $\lambda\int_{\Omega}v^2~dx-\int_{\Omega}|\nabla v|^2~dx\geq \lambda\int_{\Omega}v^2~dx-\sum\limits_{i=1}^{j}\int_{\Omega}\lambda_i| e_i|^2~dx\geq (\lambda-\lambda_j)\int_{\Omega}|v|^2~dx > 0$. Hence, $\rho(v)v\in \mathcal{N}_{\lambda},$ and then $\rho(S_j)\in\Sigma',$ and $\gamma(\rho(S_j))=\gamma(S_j)=j$ for $1\leq j\leq k,$ for any $k\in\mathbb{N}^*$.\\
It is then standard ([see \cite{AM}, Proposition 10.8]) to conclude that $$\sigma_{\lambda,j}=\inf_{\gamma(A)\geq j}\sup_{u\in A} I_{\lambda}(u),~~1\leq j\leq k,~~\text{for any }~k\in\mathbb{N}^*$$ yields $k$ pairs of nontrivial critical points for $I_{\lambda},$ which gives rise to $k$ nontrivial solutions of problem (\ref{e3}).\\
\\
Case 2: $ p>2$.\\
In this case, we will rely on the following theorem.\\
\textbf{Theorem} \textup{(Clark, \cite{Cl}) \label{cl1}}.\\
\textup{Let $X$ be a Banach space and $G\in C^1(X,\mathbb{R})$ satisfying the Palais-Smale condition with $G(0)=0.$ Let $\Gamma_k =\{~A\in\Sigma~:~\gamma(A)\geq k~\}$ with $\Sigma= \{~A\subset X~;~A=-A~\text{and}~A~\text{closed}~ \}.$ If $c_k=\inf\limits_{A\in \Gamma_k}\sup\limits_{u\in A}G(u)\in (-\infty, 0),$ then $c_k$ is a critical value.}\\

\noindent Let us consider the $C^1$ energy functional  $I_{\lambda}: W^{1,p}_0(\Omega)\rightarrow\mathbb{R}$ defined as $$I_{\lambda}(u)=\frac{2}{p}\int_{\Omega}|\nabla u|^p~dx+\int_{\Omega}|\nabla u|^2~dx-\lambda\int_{\Omega}|u|^2~dx.$$ 
We want to show that 
\begin{equation}\label{E5}
-\infty<\sigma_j=\inf_{\{A\in\Sigma', \gamma(A)\geq j\}}\sup_{u\in A}I_{\lambda}(u)
\end{equation} 
is a critical point for $I_{\lambda}$, where $\Sigma'=\{A\subseteq S_j\},$ where $S_j=\{v\in E_j~:~\int_{\Omega}|v|^2~dx=1\}.$ \\
We clearly have that $I_{\lambda}(u)$ is  an even functional for all $u\in W^{1,p}_0(\Omega)$, and also $I_{\lambda}(u)$ is bounded from below on $W^{1,p}_0(\Omega)$ since $I_{\lambda}(u)\geq C\|u\|_{1,p}^p-C'\|u\|^2_{1,p}$. \\
\\
We show that $I_{\lambda}(u)$ satisfies the (PS) condition. Let $\{u_n\}$ be a Palais-Smale sequence, i.e., $|I_{\lambda}(u_n)|\leq M$ for all $n,$ $M>0$ and $I_{\lambda}'(u_n)\rightarrow 0$ in $W^{-1,p'}(\Omega)$ as $n\rightarrow\infty.$  We first show that $\{u_n\}$ is bounded in $W^{1,p}_0(\Omega).$ We have 
\begin{eqnarray*}
M&\geq & |C\|u_n\|_{1,p}^p-C'\|u_n\|^2_{1,p}|\geq \left( C\|u_n\|_{1,p}^{p-2}-C'\right) \|u_n\|_{1,p}^2,
\end{eqnarray*}
and so $\{u_n\}$ is bounded in $W^{1,p}_0(\Omega).$
Therefore, $u\in W^{1,p}_0(\Omega)$ exists such that, up to subsequences that we will denote by $(u_n)_n$ we have $u_n\rightharpoonup u$ in $W^{1,p}_0(\Omega)$ and $u_n\rightarrow u$ in $L^2(\Omega).$\\
We will use the following inequality for $v_1,v_2\in\mathbb{R}^N:$ there exists $R>0$ such that
$$|v_1-v_2|^p\leq R\left(|v_1|^{p-2}v_1-|v_2|^{p-2}v_2\right)(v_1-v_2),$$ for~$p>2$ [\text{see} \cite{PL1}]. Then we obtain
\begin{eqnarray*}
\langle I'_{\lambda}(u_n)-I'_{\lambda}(u),u_n-u\rangle &=&2\int_{\Omega}\left(|\nabla u_n|^{p-2}\nabla u_n-|\nabla u|^{p-2}\nabla u\right)\cdot \nabla(u_n-u)~dx+2\int_{\Omega}|\nabla u_n-\nabla u|^2~dx\\ &-&2\lambda\int_{\Omega}|u_n-u|^2~dx\\
&\geq & \frac{2}{R}\int_{\Omega}|\nabla u_n-\nabla u|^p~dx +2 \int_{\Omega}|\nabla u_n-\nabla u|^2~dx-2\lambda\int_{\Omega}|u_n-u|^2~dx\\
&\geq & \frac{2}{R}\|u_n-u\|^p_{1,p}-2\lambda\int_{\Omega}|u_n-u|^2~dx.
\end{eqnarray*}
Therefore $\|u_n-u\|_{1,p}\rightarrow 0$ as $n\rightarrow +\infty$, and so $u_n$ converges to $u$ in $W^{1,p}_0(\Omega).$\\
\\
Next, we show that there exists sets $A_j$ of genus $j=1,\dots,k$ such that $\sup\limits_{u\in A_j}I_{\lambda}(u)<0.$\\
\\
Consider $E_j=\text{span}\{e_i,~i=1,\dots,j\}$ and $S_j=\{v\in E_j:~\int_{\Omega}|v|^2~dx=1\}.$ For any $s\in (0,1)$, we define the set $A_j(s):=s(S_j\cap E_j)$ and so $\gamma(A_j(s))=j$ for $j=1,\dots,k.$ We have, for any $s\in (0,1)$
\begin{eqnarray}\notag
\sup\limits_{u\in A_j}I_{\lambda}(u) &= & \sup\limits_{v\in S_j\cap E_j}I_{\lambda}(sv)\\\notag
&\leq& \sup\limits_{v\in S_j\cap E_j}\left\{\frac{s^p}{p}\int_{\Omega}|\nabla v|^pdx+\frac{s^2}{2}\int_{\Omega}|\nabla v|^2dx-\frac{\lambda s^2}{2}\int_{\Omega}|v|^2dx\right\}\\
&\leq & \sup\limits_{v\in S_j\cap E_j}\left\{\frac{s^p}{p}\int_{\Omega}|\nabla v|^pdx+\frac{s^2}{2}(\lambda_j-\lambda)\right\}<0
\end{eqnarray}
for $s>0$ sufficiently small, since $\displaystyle{\int_{\Omega}}|\nabla v|^p~dx\leq c_j$, where $c_j$ denotes some positive constant.\\
Finally, we conclude that $\sigma_{\lambda,j}$ $(j=1,\dots,k)$ are critical values thanks to Clark's Theorem.
\end{proof}

\end{document}